\documentclass[onefignum,onetabnum]{siamonline220329}

\usepackage{lipsum}
\usepackage{amsfonts}
\usepackage{graphicx}
\usepackage{amssymb}
\usepackage{epstopdf}
\usepackage{algorithmic}
\usepackage{mathtools}
\usepackage{bbm}
\usepackage{booktabs}
\usepackage{todonotes}
\usepackage{multirow}
\usepackage{makecell}
\usepackage{setspace}
\usepackage[export]{adjustbox}
\usepackage{cases}
\usepackage[createShortEnv]{proof-at-the-end}
\usepackage[shortlabels]{enumitem}
\ifpdf
  \DeclareGraphicsExtensions{.eps,.pdf,.png,.jpg}
\else
  \DeclareGraphicsExtensions{.eps}
\fi

\usepackage{subcaption}
\usepackage{tikz}
\usetikzlibrary{intersections}
\usetikzlibrary{angles, quotes}
\newcommand{\R}{\mathbb{R}}
\newcommand{\1}{\mathbbm{1}}
\newcommand{\N}{\mathbb{N}}

\newcommand{\Nc}{\mathcal{N}}

\newcommand{\Wc}{\mathcal{W}}

\newcommand{\Oc}{\mathcal{O}}

\renewcommand{\d}{\text{d}}

\newcommand{\Bc}{\mathcal{B}}

\newcommand{\Ec}{\mathcal{E}}
\newcommand{\Ex}{\mathbb{E}}
\newcommand{\Hc}{\mathcal{H}}
\newcommand{\Pc}{\mathcal{P}}
\newcommand{\Fc}{\mathcal{F}}

\DeclareMathOperator{\id}{I}

\newcommand{\law}{\mathrm{law}}

\DeclareMathOperator*{\argmin}{argmin}

\DeclareMathOperator*{\prox}{prox}

\DeclareMathOperator*{\dive}{div}

\DeclareMathOperator*{\TV}{TV}

\DeclareMathOperator*{\KL}{KL}

\DeclareMathOperator{\E}{\mathbb{E}}
\let\P\relax
\DeclareMathOperator{\P}{\mathbb{P}}

\usepackage{enumitem}
\setlist[enumerate]{leftmargin=.5in}
\setlist[itemize]{leftmargin=.5in}

\newsiamremark{remark}{Remark}
\newsiamremark{example}{Example}

\newcounter{assCounter}

\newsiamthm{assInner}{Assumption} 
\newenvironment{ass}
  {\refstepcounter{assCounter}%
   \begin{assInner}}
  {\end{assInner}}

\crefname{example}{Example}{Example}
\newsiamthm{claim}{Claim}

\headers{Langevin Dynamics for Non-smooth Potentials}{L. Frühwirth, A. Habring}

\title{Ergodicity of Langevin Dynamics and its Discretizations for Non-smooth Potentials with Linearly Growing Drift}

\author{Lorenz Frühwirth\thanks{Faculty of Computer Science and Mathematics, University of Passau, Germany (\email{lorenz.fruehwirth@uni-passau.de}).}
\and
Andreas Habring\thanks{Institute of Visual Computing, Graz University of Technology, Austria (\email{andreas.habring@tugraz.at}).}}

\usepackage{amsopn}

\renewcommand{\phi}{\varphi}
\renewcommand{\epsilon}{\varepsilon}

\ifpdf
\hypersetup{
  pdftitle={Subgradient Methods for Langevin Sampling from Non-smooth Potentials},
  pdfauthor={}
}

\usepackage{pdfpages} 
\usepackage{pgffor} 

\makeatletter
\AtBeginDocument{\let\LS@rot\@undefined}
\makeatother

\def\supplementfilename{supplement.pdf_}
\pdfximage{\supplementfilename}
\def\numbersupplementpages{\the\pdflastximagepages}

\newif\ifarXiv
\arXivtrue 

\begin{document}

\maketitle

\begin{abstract}
This article is concerned with sampling from Gibbs distributions $\pi(x)\propto e^{-U(x)}$ using Markov chain Monte Carlo methods. In particular, we investigate Langevin dynamics in the continuous- and discrete-time setting for such distributions with potentials $U(x)$ which are strongly convex but possibly non-differentiable and whose subgradients grow at most linearly. We show that the corresponding subgradient Langevin dynamics are exponentially ergodic with invariant distribution $\pi$ in the continuous setting and that its explicit as well as a semi-implicit discretization are geometrically ergodic, the semi-implicit one exhibiting better convergence behavior for vanishing step-sizes.
We show that both discretizations approximate the target distribution $\pi$ for vanishing discretization step size.  Moreover, we prove that the discrete schemes satisfy the law of large numbers allowing us to use consecutive iterates in order to compute statistics of the stationary distribution instead of simulating multiple independent Markov chains.
Numerical experiments are provided confirming the theoretical findings and showcasing the practical relevance of the proposed methods in imaging applications.
\end{abstract}
\vspace{0.2cm}

\section{Introduction}
This article is concerned with the use of Langevin dynamics for sampling from \emph{Gibbs} probability distributions of the form 
\begin{equation}\label{eq:target}
    \pi(x) = \frac{e^{-U(x)}}{Z}
\end{equation}
where $Z=\int e^{-U(y)}\d y$ is a mere normalization constant and the potential $U:\R^d\rightarrow \R$ is strongly convex but possibly non-differentiable, and such that its subgradient grows at most linearly (see \Cref{Ass:Potentials}). Sampling from Gibbs measures of the form \eqref{eq:target} is a task frequently arising in various applications such as Bayesian inference in mathematical imaging, see for instance \cite{pereyra2016proximal,durmus2022proximal,ehrhardt2023proximal,zach2022stable,zach2022computed,narnhofer2022posterior,purohit2024posterior}, or in the context of diffusion models as in \cite{guanxiong2023bayesian,chung2022score,song2020score,purohit2024posterior}. In particular, we will also consider the special case of potentials factoring as $U(x) = F(x) + G(x)$ with $F$ and $G$ having standard properties frequently encountered especially in the field of inverse problems where typically $F$ denotes the data fidelity term and $G$ denotes a regularization functional, see for example, \cite{habring2023subgradient,narnhofer2022posterior,pereyra2016proximal,ehrhardt2023proximal}. In this article we investigate the continuous-time Langevin dynamics governed by the potential $U(x)$ along with the Markov chains obtained from explicit and semi-implicit discretizations of these dynamics. In particular, we establish geometric ergodicity in both the continuous- and discrete-time setting. Moreover, we show that the suggested discretizations satisfy a law of large numbers (LLN) which renders these schemes particularly relevant for applications: If an LLN holds, consecutive iterates of a \emph{single} Markov chain can be used as a sample for the computation of statistics of the stationary distribution. In the converse case, without an LLN, an approximating sample of the stationary distribution can only be obtained by simulating several independent Markov chains which constitutes a significant increase in memory consumption and numerical effort. This is especially crucial for high dimensional applications such as imaging.
A popular strategy for obtaining samples from $\pi$ is to use Markov chain Monte Carlo (MCMC) methods, where a Markov chain $(X_k)_{k \in \N}$ is generated such that the law of $X_k$ approximates $\pi$ as $k\rightarrow \infty$. A possible way of obtaining such Markov chains is to discretize the over-damped Langevin diffusion SDE as proposed by \cite{rossky1978brownian,parisi1981correlation}
\begin{equation}\label{eq:intro:Langevin_SDE}
\d Y_t = -\nabla U(Y_t)\;\d t + \sqrt{2}\;\d B_t
\end{equation}
with $B_t$ denoting Brownian motion. Note that \eqref{eq:intro:Langevin_SDE} does not rely on knowledge of the normalization constant $Z$, which is hard to compute in practice. Under sufficient regularity conditions on $U$, respectively $\nabla U$, \eqref{eq:intro:Langevin_SDE} is ergodic with stationary distribution $\pi$ in total variation norm \cite[Theorem 2.1]{roberts1996exponential}, or in Wasserstein distance, see \cite{bolley2012convergence}. A discrete time Markov chain (MC) approximating \eqref{eq:intro:Langevin_SDE} is obtained, for example, via the Euler-Maruyama discretization, referred to as the unadjusted Langevin algorithm (ULA)
\begin{equation}
\label{eq:intro:Langevin_discr}
X_{k+1} =X_k -\tau_{k+1} \nabla U(X_{k}) + \sqrt{2\tau_{k+1}}\; B_{k+1}
\end{equation}
where $\tau_k>0$ is the step size and $(B_k)_{k \in \N}$ an i.i.d. sequence of $d$-dimensional standard Gaussian random variables. If $\tau_k=\tau$ for all $k$, $(X_k)_k$ admits a stationary distribution $\pi^\tau$ which is, in general, not equal to $\pi$, but approximates $\pi$ in the sense that $\pi^\tau\rightarrow \pi$ as $\tau\rightarrow 0$, cf.~\cite{talay1990, durmus2019analysis, mattingly2002}. For $\tau_k \rightarrow 0$ with $\sum_k \tau_k =\infty$, however, a direct approximation of $\pi$ without a remaining bias is feasible, see \cite{lamberton2003recursive,habring2023subgradient}. Non-asymptotic convergence bounds have been shown in the total variation norm~\cite{durmus2017nonasymptotic,dalalyan2017theoretical}, as well as in Wasserstein distance~\cite{durmus2019high,durmus2019analysis,habring2023subgradient}.

In this article, we focus on extending ergodicity and convergence results of Langevin-type methods to a broader class of non-differentiable potentials, for which such results have not yet been available. In particular we make the following contributions.
\subsection{Contributions}
\begin{enumerate}
    \item We show that, for potentially non-differentiable but strongly convex potentials $U$ whose subgradient grows at most linearly, the continuous-time Langevin diffusion is exponentially ergodic with stationary distribution $\pi$.
    \item We consider two different discretization schemes of the diffusion: an explicit subgradient scheme and, for potentials of the form $U(x)=F(x)+G(x)$, a semi-implicit proximal-gradient scheme. For both, we show geometric ergodicity when using a fixed step size $\tau$ (see~\Cref{prop:explicit_scheme_ergodic,prop:semi_implicit_fixed_point}) with the semi-implicit scheme yielding better convergence rates for $\tau$ close to zero (see~\Cref{prop:semi_implicit_fixed_point}). Moreover, we show convergence of the stationary distribution $\pi^\tau$ of the discrete schemes to the target distribution $\pi$ as $\tau \rightarrow 0$ with Wasserstein-2 error between $\pi$ and $\pi^\tau$ of order $\tau$ in~\Cref{prop:explicit_convergence_of_stationary,prop:implicit_final_convergence} (cf.~\cite{MALA_BRH_2013,Minorization_DEMS_2025}). To the best of our knowledge such ergodicity results in the non-differentiable (and non-Lipschitz continuous) case have not been available in the literature as previous works have assumed higher regularity~\cite{durmus2017nonasymptotic,durmus2019high,roberts1996exponential}, they have not shown ergodicity of the schemes~\cite{ehrhardt2023proximal,habring2023subgradient,durmus2019analysis}\footnote{there the authors show convergence of the ergodic means of the iterate-distributions without ergodicity of the scheme itself under the additional assumption of Lipschitz continuity}, or they have circumvented non-differentiability by prior smoothing~\cite{pereyra2016proximal,durmus2022proximal}.
    Specifically, contrary to previous works, we do neither require Lipschitz continuity of $U(x)$ (or $G(x)$ in the case $U(x) = F(x)+G(x)$) nor of $\nabla U(x)$. Instead we reduce the regularity requirements to at most linear growth of the subgradient $\partial U$.
    \item As a practically relevant contribution we show that the chains satisfy the law of large numbers allowing to use consecutive iterates of the chain without any thinning \cite[Remark 4.1]{narnhofer2022posterior} for the computation of statistics of the stationary distribution 
    which may result in a decrease in memory consumption in practice. 
    \item We present numerical experiments for potentials $U(x)$, where direct sampling (without prior smoothing) had not been possible before, confirming the theoretical results provided in the paper. Moreover, we conduct experiments for inverse imaging applications, namely, image denoising and deconvolution.
\end{enumerate}
\subsection{Organization of the article}
In \Cref{sec:related} we discuss related work for sampling from non-smooth potentials, in \Cref{sec:preliminaries} we introduce some relevant notation and terminology, in \Cref{sec:continuous_analysis} we analyze the continuous-time Langevin diffusion and in \Cref{sec:discrete_analysis} an explicit and a semi-implicit discretization. In  \Cref{sec:experiments} we show numerical results supporting the theory provided in the paper and in \Cref{sec:conclusion} we finish with a discussion of the paper.

\section{Related Work}\label{sec:related}
Conventional ULA \eqref{eq:intro:Langevin_discr} for sampling from Gibbs distributions relies on differentiability of the potential $U$, see, for instance, \cite{roberts1996exponential}. Since many applications, such as regression with $L^1$ loss or Lasso priors, as well as non-smooth regularization functionals in Bayesian imaging and inverse problems, lead to non-differentiable potentials, substantial research has been devoted to extending sampling methods to these settings.

\subsection{Langevin sampling from non-smooth potentials}
In~\cite{durmus2019analysis} the authors propose a subgradient as well as a proximal gradient method for sampling from non-smooth potentials. The subgradient method from~\cite{durmus2019analysis} is adapted in~\cite{habring2023subgradient} to allow for improved convergence results. In both works, it is shown that the iterates of the MC will be close to the target, however, none of these works establishes ergodicity---that is, convergence to a stationary distribution---of the iterates. In~\cite{ehrhardt2023proximal} it is shown that for a proximal-gradient scheme the convergence can be maintained despite computing the proximal mapping only approximately, thus, allowing also for its iterative computation within the algorithm. The work~\cite{salim2020primal} provides a duality result for proximal-gradient sampling viewed as an optimization problem in the space of probability measures. Non-smooth potentials might also be handled by adapting primal-dual algorithms from optimization to sampling. This approach was first proposed by \cite{narnhofer2022posterior} and has recently been thoroughly analyzed by \cite{burger2024coupling}. While providing valuable theoretical insights, unfortunately it turns out that the smoothness requirements on the involved functions are still rather strong. In~\cite{Sabanis_Johnston_ULA} Langevin sampling from non-smooth but strongly convex potentials is considered under the assumption that the points of non-differentiability are concentrated on suitable manifolds. 

\subsection{Potential smoothing and similar techniques}
Other works circumvent the issue of non-differentiability by smoothly approximating non-smooth parts of the potential. For the smooth approximation convergence results can be transferred from the classical literature and the task reduces to showing that the surrogate density, indeed, approximates the target. \cite{durmus2022proximal,pereyra2016proximal,brosse2017sampling} replace non-smooth parts of the density by its Moreau envelope. Expressing the gradient of the Moreau envelope via the proximal mapping, this approach effectively leads to a proximal-gradient sampling algorithm. \cite{durmus2022proximal,brosse2017sampling} show that the obtained smooth approximation of the target density can be made arbitrarily accurate in total variation (TV) by decreasing the Moreau-Yoshida parameter. (Note, however, that decreasing this parameter makes the method successively slower, see \Cref{table:computation_times_2d}.) The resulting method is coined MYULA (Moreau Yoshida unadjusted Langevin algorithm). In P-MALA (proximal Metropolis adjusted Langevin algorithm), see, for instance, \cite{pereyra2016proximal,cai2022proximal,luu2021sampling}, similarly, the proximal mapping is employed, however, convergence is assured by introducing a Metropolis-Hastings correction step within each iteration of the method. A general downside of such methods dealing with non-smooth potentials via proximal mappings is that the computational burden might increase significantly if the prox is not explicit which is frequently the case in practice. In contrast, an explicit scheme such as presented in \Cref{sec:explicit_scheme} can be used without additional computational effort as long as the subgradient of $U(x)$ is explicit. \cite{laumont2022bayesian} consider smoothing the target density before sampling similarly to MYULA. This time, however, the smoothing is achieved by convolving the non-smooth part with a Gaussian kernel. The gradient of the smoothed potential can be computed in special cases using Tweedie's formula which leads to a Plug and Play approach.

\subsection{Methods not employing Langevin dynamics}
We also want to mention some approaches for sampling which do not employ Langevin dynamics. \cite{liang2022proximal,chen2022improved,lee2021structured} consider the regularized density $\tilde{\pi}(x,y)\propto\exp(-U(x) - \mu\|x-y\|^2)$ from which they sample using a Gibbs algorithm, that is, sampling alternatingly, once from the regularized Gaussian distribution $\tilde{\pi}(x|y)$, which is referred to as a \emph{restricted Gaussian oracle}, and once from the simple Gaussian $\tilde{\pi}(y|x)$. This framework is cast as alternating sampling framework (ASF). Note that the marginal $\tilde{\pi}(x) = \int \tilde{\pi}(x,y)\;\d y \propto \exp(-U(x)) \int \exp(- \mu\|x-y\|^2)\;\d y \propto \pi(x)$ is in fact equal to the target density since the latter integral is independent of $x$. The restricted Gaussian oracle is implemented via a rejection sampling subroutine. A downside of these alternating sampling methods is the requirement of subroutines for the computation of the restricted Gaussian oracle which renders iterations computationally expensive. A similar related line of works are so-called \emph{asymptotically exact data augmentation models} \cite{vono2021asymptotically,rendell2021global} where a joint distribution $\tilde{\pi}_\rho(x,y)$ depending on a parameter $\rho$ is introduced. Instead of requiring the marginal distribution to be equal to the target, however, it is only required that $\tilde{\pi}_\rho(x)\rightarrow\pi(x)$ as $\rho\rightarrow0$ for all $x$. In this context, \cite{vono2019split} consider in an ADMM fashion for $\pi(x)\propto \exp(-F(x)-G(x))$ the approximation $\tilde{\pi}_\rho(x,y)=\exp(-F(x)-G(y)-\phi(x,y;\rho))$ with a suitable \emph{distance} function such that $\tilde{\pi}_\rho(x)\rightarrow\pi(x)$. Sampling is then again performed using a Gibbs sampler. Similarly, \cite{vono2022efficient} approximate distributions of the form $\pi(x)\propto \exp(-U(Kx))$ with a linear operator $K$ by $\tilde{\pi}_\rho(x,y)=\exp(-U(y)-\|Kx-y\|^2_2/(2\rho^2))$ such that a decoupling of operator and functional is achieved, for which Gibbs sampling is feasible. As a subroutine for sampling from the conditional density, \cite{vono2019split,vono2022efficient} propose to use MYULA, P-MALA, or rejection sampling.

\section{Notation and Preliminaries}\label{sec:preliminaries}
For $x,y\in\R^d$, $\langle x , y \rangle$ denotes the Euclidean inner product and $\|x\| = \sqrt{\langle x, x\rangle}$ the Euclidean norm. 
\paragraph{Notions of differentiability} For a convex function $\phi:\R^d\rightarrow (-\infty,\infty]$ we define the subdifferential as the (possibly empty) set
\[
\partial \phi(x) = \left\{y\in\R^d\;\middle|\; \forall h\in\R^d:\; \phi(x)+\langle y, h\rangle\leq \phi(x+h)\right\}.
\]
We use $\nabla \phi(x)$ to denote both the classical and weak gradient, with $\nabla_i \phi(x)$ its $i$-th entry. The weak gradient in this context is defined as the unique function $g\in L^1_{\mathrm{loc}}(\R^d)$ such that for all $\psi \in \mathcal{C}_c^{\infty}(\R^d)$ it holds $\int_{\R^d} \nabla_i \psi(x) \phi(x) \d x = - \int_{\R^d} \psi(x)g(x)\d x$.
For $d=1$ we will also write $\phi'(x)$ to denote the classical derivative. For any domain $\Omega \subset\R^d$ we denote the Sobolev space 
\[
H^1(\Omega)\coloneqq\left\{\phi:\Omega\rightarrow \R\;\middle|\; \text{$\phi$ admits a weak derivative with $\nabla\phi\in L^2(\Omega)$}\right\}
\]
and $H^1_\text{loc}(\R^d)\coloneqq\left\{\phi:\R^d\rightarrow \R\;\middle|\; \phi\in H^1(\Omega)\text{ for any bounded $\Omega\subset\R^d$} \right\}$. We further define the proximal mapping of a proper, convex, and lower semi-continuous function $\phi$ as
\begin{equation}\label{eq:defn_prox}
{\prox}_{\phi}(x) = \argmin_{z\in\R^d} \phi(z) + \frac{1}{2} \|x-z\|^2 = (I+\partial \phi)^{-1}(x).
\end{equation}
\paragraph{Probability measures and Wasserstein distance} 
We denote the Borel $\sigma$-algebra on $\R^d$ as $\Bc(\R^d)$ and the space of probability measures on $(\R^d,\Bc(\R^d))$ as $\Pc(\R^d)$. Moreover, for $p > 0$, let $\Pc_p(\R^d)$ denote the subspace of probability measures with finite $p$-th moment, that is, $\int_{\R^d} \| x\|^p \;\d \mu(x) < \infty$. For a measurable function $f:\R^d\rightarrow \R$ and a probability measure $\mu\in\Pc(\R^d)$, we denote $\mu(f)\coloneqq\int f(x)\;\d \mu(x)$. A coupling of two probability measures $\mu, \nu \in \Pc( \R^d)$ is a probability measure $\zeta \in \Pc( \R^d \times \R^d)$ with marginal distributions $\mu$ and $\nu$, i.e., for $A\in\Bc(\R^d)$, $\zeta(A\times\R^d)=\mu(A)$ and $\zeta(\R^d\times A)=\nu(A)$. We denote the set of all couplings of $\mu$ and $\nu$ as $\Pi(\mu,\nu)$. In an abuse of terminology, we will also refer to a couple of random variables $(X,Y)$ as a coupling of $\mu$ and $\nu$ if the joint distribution of $(X,Y)$ is an element of $\Pi(\mu,\nu)$. The (Kantorovich) Wasserstein $p$-distance between $\mu,\nu\in\Pc_p(\R^d)$ is defined as
\begin{equation*}
\Wc_p(\mu,\nu) = \left(\inf\limits_{\zeta\in\Pi(\mu,\nu)} \int \|x-y\|^p\;\d \zeta(x,y)\right)^{1/p} = \left(\inf\limits_{\substack{X\sim\mu\\Y\sim\nu}} \E[\|X-Y\|^p]\right)^{1/p}.
\end{equation*}
The infimum in the definition of the Wasserstein-$2$ distance is, in fact, attained \cite{villani2009optimal} and we will refer to the minimizer as an optimal coupling for $\mu,\nu$.

\paragraph{Markov chains, Markov processes, and ergodicity}
We call a function $P:\R^d\times\Bc(\R^d)\rightarrow [0,1]$ a Markov transition kernel \cite[Section 3.4.1]{meyn2012markov} if for each $x$, $P(x,\;\cdot\;)$ is a probability measure and for each $A\in\Bc(\R^d)$, $P(\;\cdot\;,A)$ is measurable. A sequence of random variables $(X_k)_{k \in \N}$ on $\R^d$ is called a (time-homogeneous) Markov chain (MC) with transition kernel $P$ and initial distribution $\mu_0\in\Pc(\R^d)$ if for any $n\in \N$, and $A_0,\dots A_n\in\Bc(\R^d)$ it holds
\begin{equation*}
    \begin{aligned}
        P[X_0\in A_0,\dots,& \;X_n\in A_n]  = \\
& \int\limits_{x_0\in A_0}\dots   \int\limits_{x_{n-1}\in A_{n-1}}  P(x_{n-1},A_n) P(x_{n-2}, \d x_{n-1}) \ldots  P(x_0,\d x_1) \d\mu_0(x_0).
    \end{aligned}
\end{equation*}
A distribution $\mu_\infty\in\Pc(\R^d)$ is called \emph{stationary} or \emph{invariant} for the MC if
\[
\mu_\infty(A) = \int P(x,A)\;\d\mu_\infty(x),\quad A\in\Bc(\R^d)
\]
and we say the chain is ergodic if there exists such a stationary distribution $\mu_\infty$ and $\mu_k$, the law of $X_k$, converges to this stationary distribution as $k\rightarrow\infty$ in an appropriate sense. The chain is called geometrically ergodic if for all $k$, $d(\mu_k,\mu_\infty)\leq c\lambda^k$ for some $c\geq 0$, $\lambda\in (0,1)$ and $d(\;\cdot\;,\;\cdot\;)$ a suitable metric on the space of probability measures. In the continuous-time setting these notions can be generalized as follows: A (time-homogeneous) transition function $(P_t)_{t\geq 0}$ is a family of maps, such that for each $t\geq 0$, $P_t$ is a Markov transition kernel and the family satisfies the Chapman-Kolmogorov equation \cite[Definition 1.5]{sharpe1988general}
\[
P_{s+t}(x,A) = \int P_s(y,A)\; P_t(x,\d y),\quad A\in\Bc(\R^d),\;s,t\geq 0.
\]
The family of random variables $(X_t)_{t \geq 0}$ on $\R^d$ is a (time-homogeneous) Markov process with transition function $(P_t)_{t \geq 0}$ if for any bounded and measurable $f:\R^d\rightarrow \R$
\[
\E[f(X_{t+s})\;|\;\Fc_s] = \int f(y) \;P_t(X_s,\d y),\quad 0\leq s,t<\infty
\]
where $\Fc_t = \sigma\left(\left\{ X_s\;\middle|\; 0\leq s\leq t\right\}\right)$ is the $\sigma$-algebra generated by the process up until time $t$.
For a Markov process $(X_t)_{t\geq 0}$ on $\R^d$ with transition function $(P_t)_{t\geq 0}$ a probability distribution $\mu_\infty\in\Pc(\R^d)$ is called \emph{stationary} or \emph{invariant} if it satisfies
\[
\mu_\infty(A) = \int P_t(x,A)\; \d \mu_\infty(x),\quad A\in\Bc(\R^d),\;t\geq 0.
\]
As before, we say the process $(X_t)_{t \in [0, \infty)}$ is \emph{ergodic} if there exists such a stationary distribution $\mu_\infty$ and $\mu_t$, the law of $X_t$, converges to this stationary distribution as $t\rightarrow\infty$. We call the process exponentially ergodic if for all $t\geq 0$, $d(\mu_t,\mu_\infty)\leq c\lambda^t$ some $c\geq 0$, $\lambda\in (0,1)$.

\section{Analysis of the Continuous-time Dynamics}\label{sec:continuous_analysis}
Our theoretical analysis starts with investigating the continuous-time Langevin dynamics governed by the potential $U$. That is, we will consider Ito-processes $(X_t)_{t \geq 0}$ satisfying the stochastic differential equation (SDE)
\begin{equation}
\begin{cases}
\label{eq:sde}
\d X_t& = -\nabla U(X_t) dt + \sqrt{2} \d W_t, \\
\phantom{\d} X_0  & \sim  \mu
\end{cases}
\end{equation}
where $(W_t)_{t \geq 0}$ is a standard $d$-dimensional Brownian motion, $ \mu \in \mathcal{P}_2( \R^d)$, and $\nabla U$ is the weak derivative, which exists under the imposed assumptions on $U$. Under \Cref{Ass:Potentials}, we establish the existence and uniqueness of a solution to \eqref{eq:sde}, as well as the ergodicity with unique stationary distribution $\pi$. We denote the distribution of a solution $X_t$ as $\mu_t$.
\begin{ass}
\label{Ass:Potentials}
The potential $U : \R^d \rightarrow \R$ satisfies the following conditions:
\begin{enumerate}[label=(\roman*)]
    \item $m$-strong convexity: there exists $m>0$ such that for any $x,y\in\R^d$, $\lambda\in[0,1]$
    \[
    U(\lambda x + (1-\lambda) y )\leq \lambda U(x) + (1-\lambda)U(y) - \frac{m}{2}\lambda(1-\lambda)\|y-x\|^2.
    \]
    \item At most linear growth of the subgradient: There exists $C>0$ such that
    \[
    \| g \| \leq C \left( 1 + \|x\| \right),\quad x\in\R^d,\; g\in\partial U(x).
    \]
\end{enumerate}
\end{ass}
\begin{remark}\
\label{rem:Ass_A}
\begin{itemize}
    \item Strongly convex functions on $\R^d$ are coercive and continuous. Thus, $U(x)$ admits a minimizer which we can assume to be at $x=0$ without loss of generality. (In the converse case, sample from the potential $\hat{U}(x) = U(x+\hat{x})$ with $\hat{x}$ the minimizer and afterwards, add $\hat{x}$ to the obtained samples.) 
    \item Since $U$ is a convex function it follows immediately that $U \in H^1_{\text{loc}}( \R^d)$. Moreover, by Alexandrov's theorem, $U$ is (twice) differentiable in the classical sense on a set whose complement has Lebesgue measure zero. On that set the classical derivative of $U$ coincides with the weak derivative. On the remaining zero-set where this is not the case, one may select an arbitrary subgradient.
    \item The linear growth of the subgradient together with convexity immediately implies quadratic growth of $U$. Indeed, we have for $g\in\partial U(x)$ and some $c>0$
    \begin{equation}
        | U(x) | \leq |U(0)| + |\langle g,x\rangle |\leq |U(0)| + \|g\|\|x\|\leq c(1+\|x\|^2).
    \end{equation}
    \item By $m$-strong convexity of $U$, there exists a convex function $c(x)$ such that $U(x) = c(x) + \frac{m}{2}\|x\|^2$. Therefore, since convex functions in turn are lower bounded by linear functions, all moments of $\pi$ are finite, that is, for any $p\geq 0$, $\int_{\R^d} \|x\|^p e^{-U(x)} dx  < \infty.$
\end{itemize}
\end{remark}

\begin{remark}\
\begin{itemize}
 \item For now, results will be proven solely for the potential $U$. Later, particularly when discussing semi-implicit schemes in \Cref{sec:semi_implicite_scheme} and in numerical experiments in \Cref{sec:experiments}, we will consider the specific case of potentials that factor as $U(x)=F(x)+G(x)$, where $F$ and $G$ each satisfy certain conditions. This setting is particularly relevant for the considered applications in inverse problems.
\item Interestingly, the linear growth condition on the drift $\nabla U(x)$ plays a pivotal role in establishing several key results presented in this work: We need at most linear growth of $\nabla U$ for guaranteeing the existence of solutions to the time-continuous SDE. On the other hand, in order to characterize the stationary distribution of \eqref{eq:sde}, we will invoke the corresponding weak Fokker-Planck equation where the linear growth of $\nabla U(x)$ plays a vital role in showing uniqueness of the solution to this PDE. Furthermore, the growth condition allows us to extend methods from \cite{durmus2019analysis,habring2023subgradient} to prove convergence of the discrete schemes without assuming $\nabla U(x)$ to be globally bounded.
\end{itemize}
\end{remark}
\subsection{Existence and uniqueness of the continuous-time process}
\label{sec:Ex_unique}
We start with existence of a solution of the SDE \eqref{eq:sde}.

\begin{theorem}\cite[Theorem 1]{zhang2005SDE}
\label{thm:Ex_Uniq_SDE}
Let $U: \R^d \rightarrow \R$ satisfy \Cref{Ass:Potentials}. Then, for any initial value $X_0\sim \mu \in \mathcal{P}_2(\R^d)$, the SDE \eqref{eq:sde} has a unique strong solution $(X_t)_{t \geq 0}$.
\end{theorem}
\begin{remark}\
\label{remark:sde_evolution_well_defined}
\begin{itemize}
    \item The uniqueness of a strong solution for any fixed initial random variable $X_0$ implies pathwise uniqueness in the sense of Definition 2 from \cite{yamada1971uniqueness}. This, in turn, implies uniqueness in law \cite[Proposition 1]{yamada1971uniqueness}. Finally, we find that, for fixed $t \geq 0$, the map
    \begin{equation*}
        \begin{aligned}
            \Pc_2(\R^d)&\rightarrow \Pc_2(\R^d)\\
            \mu &\mapsto \mu_t\coloneqq\text{Law}(X_t),\quad \text{where }
            \begin{cases}
                X_t \text{ satisfies \eqref{eq:sde}}\\
                X_0\sim\mu
            \end{cases}
        \end{aligned}
    \end{equation*}
    is well-defined by~\Cref{lemma:continuous_second_moment} and since the distribution of $X_t$ is independent of the specific choice of $X_0$ as long as the distribution of $X_0$ is fixed \cite[Corollary 2]{yamada1971uniqueness}.
    \item The family of distributions $(\mu_t)_{t > 0}$ possesses probability densities $(p_t)_{t > 0}$ (see \cite[Corollary 6.3.2]{bogachev2022fokker} together with~\Cref{rmk:FKP}), which implies that $X_t$ almost surely stays away from the set of points, where $\nabla U$ does not exist in the strong sense. 
\end{itemize}
\end{remark}

\subsection{Ergodicity of the continuous-time process}
\label{sec:Ergodicity}
In this section we prove that the solution $(X_t)_{t \geq 0}$ of the SDE \eqref{eq:sde} is exponentially ergodic. We begin by showing that the SDE solution satisfies a weak version of the Fokker-Planck equation.
\begin{lemma}\label{prop:weak_FKP}
    A solution $(X_t)_{t \geq 0}$ of the SDE \eqref{eq:sde} satisfies the following weak Fokker-Planck equation: For any $\phi\in C^2(\R^d)$ with $\int_t^{t+h} \Ex[\|\nabla \phi(X_s)\|^2\;|\; \mathcal{F}_t]\;\d s<\infty$ it holds
    \begin{equation}\label{eq:FKP1}
        \begin{aligned}
            \Ex\left[ \phi(X_{t+h})-\phi(X_t) \right]
            = \Ex\left[\int\limits_{t}^{t+h} \langle -\nabla U(X_s), \nabla \phi(X_s)\rangle + \Delta \phi(X_s)\; \d s \right]
        \end{aligned}
    \end{equation}
\end{lemma}
\begin{proof}
   Using Ito's formula \cite[Theorem 3.3]{khasminskii2012stochastic} (note that $(\nabla U(X_t))_{t \ge 0}$ is of class \textbf{L} since $\nabla U$ grows at most linearly and $(X_t)_{t \ge 0}$ is of class \textbf{L} by~\Cref{lemma:continuous_second_moment}), we find
    \begin{equation*}
            \phi(X_{t+h})-\phi(X_{t})
            = \int\limits_{t}^{t+h}\langle -\nabla U(X_s), \nabla \phi(X_s)\rangle + \Delta \phi(X_s)\; \d s + \int\limits_t^{t+h} \sqrt{2} \nabla \phi(X_s) \d W_s.
    \end{equation*}
    Taking the expectation leads to the desired result since the stochastic integral with respect to the Brownian motion $(W_t)_{t \geq 0}$ vanishes under the assumption on $\nabla \phi(X_s)$ \cite[Section 3.3]{khasminskii2012stochastic}, \cite[Chapter 8, Section 2]{gikhman1965introduction}.
\end{proof}

\begin{remark}\label{rmk:FKP}
An analogous proof using $\phi(x,t)$ with explicit time dependence yields the Fokker-Planck equation $\partial_t \mu_t(x) = \dive (\mu_t(x)\nabla U(x)) + \Delta_x \mu_t(x)$.
\end{remark}
Before proving geometric ergodicity we need two more auxiliary results.
\begin{lemma}\cite[Theorem 4.3, Chapter 2]{MAO_SDE_2011}
\label{lemma:continuous_second_moment}
    The function $t\mapsto\Ex[\|X_t\|^2]$ is continuous.
\end{lemma}
\begin{proposition}
\label{lemma:contraction}
For any two solutions of the SDE \eqref{eq:sde} $X_t \sim \mu_t$ and $\tilde{X}_t \sim \tilde{\mu}_t$ for $t\geq 0$ with initial distributions $\mu_0, \tilde{\mu}_0 \in \mathcal{P}_2(\R^d)$ it holds $\Wc_2^2(\mu_t, \tilde{\mu}_t) \leq e^{-2mt} \Wc_2^2(\mu_0, \tilde{\mu}_0)$ where $m$ is the modulus of strong convexity of $U$ from \Cref{Ass:Potentials}.
\end{proposition}
\begin{proof}
Let us denote the coupling $Z_t=(X_t,\Tilde{X}_t)$ solving the SDE
\begin{equation*}
    \begin{aligned}
            \d Z_t = \left[ \begin{array}{c}
            -\nabla U (X_t)\\
            -\nabla U (\tilde{X}_t)
            \end{array} \right]\d t + \sqrt{2} \left[ \begin{array}{c}
            \d W_t\\
            \d W_t
            \end{array} \right].
    \end{aligned}
\end{equation*}
That is, both $X$ and $\tilde{X}$ satisfy the SDE \eqref{eq:sde} with the exact same Brownian motion $(W_t)_{t \geq 0}$, however, with potentially different initial distributions. Consider $\phi(z) = \frac{1}{2}\|x-\tilde{x}\|^2$. Note that 
\[\nabla \phi (z) = \left[\begin{array}{c}
    x - \tilde{x}\\
    \tilde{x} - x
\end{array}\right], \quad (\nabla_x+\nabla_{\tilde{x}})\cdot (\nabla_x+\nabla_{\tilde{x}}) \phi (z) = 0
\]
for which it holds true that $\|\nabla \phi(z)\|^2 = 2\|x\|^2 + 2\|\tilde{x}\|^2 - 4\langle x, \tilde{x}\rangle \leq 4\|x\|^2 + 4\|\tilde{x}\|^2 = 4\|z\|^2$. Thus, $\nabla \phi (Z_t)$ is square integrable as $Z_t$ is and we can apply \Cref{prop:weak_FKP} (slightly adapted to the new drift and diffusion). Using $m$-strong convexity of $U$, this leads to
\begin{equation}
\label{eq:integral_contraction}
    \begin{split}
        \frac{1}{2}\Ex\left[\|X_{t+h} - \tilde{X}_{t+h}\|^2 - \|X_t-\tilde{X}_t\|^2 \right] 
        & = \Ex\left[ \int\limits_{t}^{t+h} \left \langle\left[
        \begin{array}{c}
        -\nabla U (X_s) \\
        -\nabla U (\tilde{X}_s)
        \end{array} 
        \right],
        \left[
        \begin{array}{c}
        X_s-\tilde{X}_s\\
        \tilde{X}_s-X_s
        \end{array} 
        \right] \right \rangle\; \d s \right] \\
        &\leq \Ex\left[ \int\limits_{t}^{t+h} -m\|X_s-\tilde{X}_s\|^2 \; \d s \right] \\
        &= - m\int\limits_{t}^{t+h} \Ex\left[ \|X_s-\tilde{X}_s\|^2 \right]\; \d s
    \end{split}
\end{equation}
for any $h>0$. Denoting $\psi(t) = \Ex\left[\|X_t-\tilde{X}_t\|^2 \right]$ we arrive at $\psi(t+h) \leq \psi(t) -2m \int_t^{t+h}\psi(s)\; \d s$. Noting that $\psi\geq 0$ it follows $2m\int_{0}^t\psi(s)\;\d s \leq \psi(0)$ for any $t\geq 0$ and, thus, $\psi(t)\rightarrow 0$ as $t\rightarrow\infty$. However, we aim for a more precise estimate. Note that continuity of $t\mapsto\Ex[\|X_t\|^2]$ (see \Cref{lemma:continuous_second_moment}) together with at most linear growth of the drift $\nabla U(x)$ implies that
\[s\mapsto \Ex\left[ \left \langle\left[\begin{array}{c}
        \nabla U (X_s)\\
        \nabla U (\tilde{X}_s)
        \end{array} \right], \left[\begin{array}{c}
        X_s-\tilde{X}_s\\
        \tilde{X}_s-X_s
        \end{array} \right] \right \rangle\right] \]
is continuous as well. Therefore, \eqref{eq:integral_contraction} implies that $\psi$ is absolutely continuous and $\psi'(t) \leq -2m \psi(t)$. Grönwall's lemma thus yields
\[\Wc_2^2(\mu_{t}, \tilde{\mu}_{t})\leq \Ex[\|X_{t}-\tilde{X}_{t}\|^2] = \psi(t)\leq \psi(0)e^{-2m t} = \Ex[\|X_{0}-\tilde{X}_0\|^2] e^{-2mt}\]
for any $t>0$. Plugging in an optimal coupling of $\mu_0$ and $\tilde{\mu}_0$ for $X_0$ and $\tilde{X}_0$ and taking the expectation  we obtain $\Wc_2^2(\mu_{t}, \tilde{\mu}_{t})\leq e^{-2mt}\Wc_2^2(\mu_0,\tilde{\mu}_0)$.
\end{proof}

\begin{theorem}
\label{thm:fixedpoint_SDE}
There exists a unique stationary distribution $\mu_{\infty} \in \mathcal{P}_2(\R^d)$ such that for any initial distribution $\mu_0 \in \mathcal{P}_2(\R^d)$ the family $(\mu_t)_{t\geq0}$ satisfies $\Wc_2^2(\mu_t, \mu_{\infty}) \leq e^{-2mt} \Wc_2^2(\mu_0, \mu_{\infty})$.
\end{theorem}
\begin{proof}
    The proof relies on an application of Banach's fixed point theorem together with methods from the theory of ergodic semigroups of operators. The uniqueness of a stationary distribution is a direct consequence of the contraction property shown in \Cref{lemma:contraction}.    
    To prove existence let us first define for $T>0$, $\Phi_T: \mathcal{P}_2(\R^d) \rightarrow \mathcal{P}_2(\R^d)$ with
    \begin{equation*}
        \begin{aligned}
            \mu \mapsto \mu_T = \text{Law}(X_T),\quad            \text{where }\begin{cases}
                (X_t)_{t \geq 0} \text{ satisfies \eqref{eq:sde}}\\
                \phantom{(} X_0 \sim \mu.
            \end{cases}
        \end{aligned}
    \end{equation*}
    which is well-defined by \Cref{remark:sde_evolution_well_defined}. By \Cref{lemma:contraction} $\Phi_T$ is a contraction with respect to the Wasserstein-2 distance $\Wc_2$. Since $(\mathcal{P}_2(\R^d),\Wc_2)$ is a complete metric space, Banach's fixed point theorem  implies that $\Phi_T$ admits a unique fixed point $\hat{\mu}_T$. This fixed point has to be a stationary measure: Note that for any two $T_1,T_2$ we find $\Phi_{T_1}(\Phi_{T_2}(\hat{\mu}_{T_1}) = \Phi_{T_2}(\hat{\mu}_{T_1})$ since the $\Phi_{T_i}$ commute. Thus, $\Phi_{T_2}(\hat{\mu}_{T_1})$ is a fixed point for $\Phi_{T_1}$ and by uniqueness of the fixed point we have $\Phi_{T_2}(\hat{\mu}_{T_1}) = \hat{\mu}_{T_1}$. This means $\hat{\mu}_{T_1}$ is a fixed point of $\Phi_{T_2}$, ergo, $\hat{\mu}_{T_1} = \hat{\mu}_{T_2}$. Therefore, there exists exactly one measure $\mu_\infty$ which is a fixed point for all $\Phi_T$, $T>0$ and thus a stationary measure. Lastly, convergence to the stationary measure from any initial distribution $\mu_0$ follows from the contraction property of the SDE.
\end{proof}

\subsection{Identifying the stationary distribution}
\label{sec:identifying_stationary_dist}
In \Cref{sec:Ex_unique,sec:Ergodicity} we have seen that the SDE \eqref{eq:sde} admits a unique strong solution $(X_t)_{t \geq 0}$ and we have established the Wasserstein-$2$ convergence $\mu_t \rightarrow \mu_{\infty}$ as $t\rightarrow\infty$ with stationary distribution $\mu_\infty$. We are left to show that this stationary distribution is, in fact, the target distribution $\pi(x)$.
\begin{theorem}
    \label{thm:identifying_stat_measure}
    The stationary distribution of the SDE \eqref{eq:sde} is $\pi(x) = \frac{e^{-U(x)}}{Z}$.
\end{theorem}
\begin{proof}
Let $(X_t)_{t \geq 0}$ solve the SDE \eqref{eq:sde} initialized at its unique stationary distribution denoted as $\mu_\infty$. Then by \Cref{prop:weak_FKP}, $(X_t)_{t \geq 0}$ satisfies for any $\phi\in C^\infty_c(\R^d)$ and all $t,h>0$
\begin{equation*}
    \begin{aligned}
        \int\limits_{t}^{t+h} \Ex\left[\langle -\nabla U(X_s), \nabla \phi(X_s)\rangle + \Delta \phi(X_s)\right]\; \d s  = 0.
    \end{aligned}
\end{equation*}
This, however, can only be true for all $t,h>0$ if the integrand is zero. We thus find that the stationary distribution $\mu_\infty$ satisfies the elliptic weak PDE
\begin{equation}
\label{eq:weak_stat_PDE}
    \begin{aligned}
        \int\limits_{\R^d} \langle -\nabla U(x) , \nabla \phi(x)\rangle  + \Delta \phi(x)\; \d \mu_\infty(x)  = 0.
    \end{aligned}
\end{equation}
Plugging in the target density $\d\pi(x)=\frac{e^{-U(x)}}{Z}\d x$ instead of $\mu_\infty$ and using the weak chain rule,
one can easily check that $\pi(x)$ is a solution to \eqref{eq:weak_stat_PDE} as well. By Example 5.1 from \cite{bogachev2002uniqueness}, this PDE admits at most one solution in the space of Borel probability measures, provided $\nabla U \in L^p_\text{loc}(\R^d)$ for a $p>d$ and if there exist $\alpha>0$ and $K\subset \R^d$ compact such that
\begin{equation}
\label{eq:uniqueness_V(x)}
    -\langle \nabla U(x),  x\rangle + d \leq \frac{\alpha}{2} \lVert x \rVert^2\quad \text{for all $x\in\R^d\setminus K$}.
\end{equation}
The above condition on $\nabla U$ is trivially satisfied by the linear growth assumption (see \Cref{Ass:Potentials}). Thus, by uniqueness $\mu_\infty = \pi$ concluding the proof.
\end{proof}

\section{Discretizations of the Langevin Dynamics}\label{sec:discrete_analysis}
\label{Sec:Discretization_algorithms}
Having analyzed the continuous-time Langevin dynamics governed by $U(x)$, the current section is devoted to providing convergent discretization schemes which enable us to sample from the target density. We will consider two different discretization schemes, namely, a fully explicit Euler-Maruyama scheme in \Cref{sec:explicit_scheme} as well as a semi-implicit scheme in \Cref{sec:semi_implicite_scheme}. The semi-implicit scheme is motivated by literature on optimization and ODEs where such schemes are known to be beneficial in terms of stability and/or convergence speed (see, e.g.,~\cite{beck2009fast}). They have also turned out to be particularly useful when dealing with non-globally Lipschitz forces $\nabla U$~\cite{kelly2022adaptive}. In addition, while the subgradient of a convex function is potentially multi-valued, its proximal mapping is unique so that implicit schemes alleviate the necessity of \emph{choosing a subgradient}. 

The discretizations considered in the following will rely on subdifferential calculus. Since $U$ is differentiable a.e., any subgradient selection is also a weak gradient so that the preceding time-continuous analysis still holds true. Our analysis shows that both Markov chains resulting from the respective discretizations are ergodic with respect to the Wasserstein-2 metric. Moreover, as a byproduct of \Cref{prop:explicit_scheme_ergodic}, for the explicit scheme\footnote{Ergodicity in total variation could also be obtained for the semi-implicit scheme by adopting the proof strategy from the explicit scheme (cf.~\cite{habring2025diffusion}). The techniques presented in this paper, however, yield explicit convergence rates and bounds on the involved constants.} we also obtain geometric convergence to the stationary distribution in total variation distance, which is more commonly used in the Markov chain literature.

\subsection{Explicit Euler-Maruyama scheme}\label{sec:explicit_scheme}
The explicit Euler-Maruyama scheme for \eqref{eq:sde} with step-size $\tau>0$ reads as
\begin{equation}
\label{eq:explicit_scheme}
    X^\tau_{k+1} = X^\tau_k -\tau \theta(X_k^\tau) + \sqrt{2\tau}W_k.
\end{equation}
with $\theta:\R^d\rightarrow \R^d$ a measurable subgradient selection\footnote{A measurable selection of subgradients exists: $U$ is almost everywhere differentiable so that only on points of non-differentiability a subgradient has to be chosen. Choosing an arbitrary subgradient yields a measurable function by completeness of the Lebesgue measure.}, i.e., $\theta$ is measurable and $\theta(x)\in \partial U(x)$ for every $x\in \R^d$.
We denote the law of $X^\tau_k$ as $\mu^\tau_k$. We begin our theoretical analysis by showing geometric ergodicity of the resulting Markov chain for fixed $\tau$. 
The proof relies on~\cite{hairer2011yet} for which we have to show a Lyapunov drift and a minorization condition (\Cref{lemma:lyapunov,lemma:minor}). 
\begin{lemma}[Lyapunov drift condition]\label{lemma:lyapunov}
    For $f(x)=\|x\|^2$, there exists $\eta>0$ independent of $\tau$ such that the discrete scheme satisfies for $\tau$ small enough the Lyapunov drift condition 
    \begin{equation*}         
        \begin{aligned}
            \Delta f(x) \coloneqq \mathbb{E} \left[ f(X_1^\tau) \middle| X_0^\tau = x \right] - f(x)
            \leq -m\tau f(x) + \eta\tau,\quad x\in\R^d
        \end{aligned}   
    \end{equation*}
\end{lemma}
\begin{proof}
    First note that by strong convexity of $U$ and since $U$ is minimal at $x=0$ (\Cref{rem:Ass_A}), $\langle\theta(x),x\rangle\geq m\|x\|^2$. Further, by linear growth of $\theta(x)$, there exist $K,c\geq 0$ such that for $\|x\|\geq K$, $\|\theta(x)\|^2\leq c\|x\|^2$. Let $C = \overline{B_{K}(0)}$ be the closed ball with center $0$ and radius $K$. Further, let $x \in C^c = \R^d \setminus C$. Then we have
    \begin{equation}\label{eq:lyapunov1}
        \begin{aligned}
            \|x-\tau\theta(x)\|^2 = \|x\|^2 - 2\tau\langle x,\theta(x)\rangle + \tau^2\|\theta(x)\|^2
            \leq & \|x\|^2 - 2m\tau\|x\|^2 + \tau^2c\|x\|^2\\
            \leq &\|x\|^2 \left( 1 - \tau \left(2m-\tau c \right)\right)\\
            \leq &\|x\|^2 \left( 1 - \tau m\right)
        \end{aligned}
    \end{equation}
    where the last inequality holds for $\tau<\frac{m}{c}$. Thus, for $x\in C^c$ we have
    \begin{equation*}         
        \begin{aligned}
            \mathbb{E} \left[ f(X_1^\tau) \middle| X_0^\tau = x \right]
            = \mathbb{E} \left[ \|x-\tau\theta(x) + \sqrt{2\tau} W_1\|^2 \middle| X_0^\tau = x \right]
            = &\|x-\tau\theta(x)\|^2 + 2\tau \\
            \leq &(1-m\tau)f(x) + 2\tau.
        \end{aligned}   
    \end{equation*}
    On the other hand, for $x\in C$ we find, since $C$ is bounded
    \begin{equation*}         
        \begin{aligned}
            \mathbb{E} \left[ f(X_1^\tau) \middle| X_0^\tau = x \right]
            = \mathbb{E} \bigg[\|x-\tau\theta(x) + \sqrt{2\tau} W_1\|^2 \bigg| &X_0^\tau = x \bigg]
            \leq  \|x\|^2 + c(K)\tau \\
            \leq & (1-m\tau)\|x\|^2 + m\tau \|x\|^2 + c(K)\tau \\
            \leq & (1-m\tau)f(x) + c \tau.
        \end{aligned}   
    \end{equation*}
\end{proof}

\begin{remark}
    While we made use of the strong convexity to prove the drift condition, we could reduce this requirement to assuming that $U$ is dissipative, that is, 
    \[
        \langle\theta(x),x\rangle\geq \alpha\|x\|^2 + b
    \]
    with $\alpha>0$, $b\in\R$. Convexity will be relevant only later to control the bias to the target distribution $\pi$.
    Assuming dissipativity, we can simply replace the estimate from \eqref{eq:lyapunov1} by
    \begin{equation*}
        \begin{aligned}
            \|x-\tau\theta(x)\|^2 = \|x\|^2 - 2\tau\langle x,\theta(x)\rangle + \tau^2\|\theta(x)\|^2
            \leq & \|x\|^2 - 2\tau\alpha \|x\|^2 + \tau^2 c\|x\|^2 -2\tau b\\
            \leq & \|x\|^2 - 2\tau\alpha\|x\|^2(1-\tau \frac{c}{\alpha}) - 2\tau b.
        \end{aligned}
    \end{equation*}
\end{remark}

\begin{lemma}[Minorization condition]\label{lemma:minor}
    Let $f(x) = \|x\|^2$ and $C=\{x\in\R^d\;|\; f(x)\leq R\}$. Then the following minorization condition holds,
    \begin{equation}\label{eq:minor}
        \inf_{x\in C}\P[X_1^\tau\in A\;|\; X_0^\tau=x]\geq \alpha(\tau) \nu_\tau(A),\quad A\in \Bc(\R^d)
    \end{equation}
    where $\nu_\tau$ denotes an isotropic Gaussian distribution on $\R^d$ with mean zero and covariance matrix $\tau I$, and $\alpha(\tau) = \frac{1}{2^{d/2}}e^{-\frac{a(R+1)}{2\tau}}$ where $a$ is a constant independent of $\tau$.
    In particular, we can choose $R>2\eta/m$ with $\eta$ from~\Cref{lemma:lyapunov} as required in~\cite[Assumption 2]{hairer2011yet}.
\end{lemma}
\begin{remark}
    We call a set $C$ such that~\eqref{eq:minor} holds $\nu_\tau$-small.
\end{remark}
\begin{proof}
    The proof is almost identical to the smooth case and relies only on local boundedness of the drift. Let $x\in C$ be arbitrary, i.e., $\|x\|^2\leq R$. We find
    \begin{equation}
        \begin{aligned}
            \P[X_1^\tau\in A\;|\; X_0^\tau=x]
             =& \int_A \frac{1}{(2\pi 2\tau)^{d/2}}e^{-\frac{\|x-\tau\theta(x) - z\|^2}{4\tau}}\d z\\
             \geq& \int_A \frac{1}{(2\pi 2\tau)^{d/2}}e^{-\frac{2\|x-\tau\theta(x)\|^2 + 2\|z\|^2}{4\tau}}\d z\\
             \geq& \frac{1}{2^{d/2}}e^{-\frac{\|x-\tau\theta(x)\|^2}{2\tau}}\int_A \frac{1}{(2\pi \tau)^{d/2}}e^{-\frac{\|z\|^2}{2\tau}}\d z.
        \end{aligned}
    \end{equation}
    By linear growth of the subgradient there exists $a$ independent\footnote{since $\tau$ is naturally bounded from above} of $\tau$ such that we have $\|x-\tau\theta(x)\|^2\leq a(R+1)$ yielding the desired result.
\end{proof}
To state the ergodicity result, let us introduce for $V:\R^d\rightarrow [1,\infty)$ the weighted TV norm
\begin{equation}
    \|\mu\|_V = \sup_{g\leq V}\int g(x)\d \mu(x) = \int V(x)\d |\mu|(x)
\end{equation}
for $\mu$ any Radon-measure on $\Bc(\R^d)$ with finite $V$-moment\footnote{For a short proof of the equality of the two representations of $\|\cdot\|_V$ we refer to the supplement.}.
\begin{theorem}\label{prop:explicit_scheme_ergodic}
    Let $\alpha(\tau),R$ be the constants from~\Cref{lemma:minor}, $\beta=\alpha(\tau)/(2\eta \tau)$, $\gamma_0 \in (1-m\tau + 2\eta\tau/R, 1)$, and denote 
    \[
        \overline{\alpha}=\max\left\{1-\alpha(\tau)/2, \frac{2+R\beta \gamma_0}{2+R\beta} \right\}.
    \]
    Then, for $\tau>0$ sufficiently small, $(X_k^\tau)_k$ is geometrically ergodic with respect to the norm $\|\;\cdot\; \|_{V}$ with $V(x)=1+\beta \|x\|^2$, that is, there exists a measure $\pi^\tau\in \Pc^2(\R^d)$ such that $\|\mu_k^\tau-\pi^\tau\|_{V} \leq \overline{\alpha}^k\|\mu^\tau_0 - \pi^\tau\|_{V}$.
    As a consequence it follows also that
    \begin{equation}
        \Wc_2^2( \mu_k^\tau, \pi^\tau) \leq \frac{2\overline{\alpha}^k}{\beta} \|\mu^\tau_0 - \pi^\tau\|_{V}.
    \end{equation}
    
\end{theorem}
\begin{remark}
    The behavior of the convergence rate as a function of $\tau$ is worse for $\tau\rightarrow 0$ for the explicit scheme compared to the semi-implicit one (cf.~\Cref{prop:semi_implicit_fixed_point}) and to results using higher regularity of $U$ (cf.~\cite{durmus2021uniform}). This is a consequence of the dependence of $\alpha(\tau)$ in the minorization condition on $\tau$. Under higher regularity it is possible to derive a uniform minorization condition leading to convergence rates as in the semi-implicit scheme, cf.~\cite{durmus2021uniform,bou2013nonasymptotic}.
\end{remark}
\begin{proof}
    By~\cite[Theorems 1.3, 3.2]{hairer2011yet}, $(X_k^\tau)_k$ admits a unique stationary measure in $\Pc_2(\R^d)$, denoted as $\pi^\tau$ and we have
    \begin{equation}
        \|\mu^\tau_k - \pi^\tau\|_{V} \leq \overline{\alpha}^k \|\mu^\tau_0 - \pi^\tau\|_{V}.
    \end{equation}
    The choice of the constants follows from ~\cite[Theorems 1.3]{hairer2011yet} together with~\Cref{lemma:lyapunov,lemma:minor}.
    Moreover, by Theorem 6.15 in \cite{villani2009optimal}, 
    \[
    \Wc_2^2( \mu_k^\tau, \pi^\tau)
    \leq 2 \int_{\R^d} \|x\|^2 \d \left|  \mu_k^\tau(x) - \pi^\tau(x) \right|
    \leq \frac{2}{\beta} \|\mu_k^\tau - \pi^\tau\|_{V}
    \]
    yielding the convergence in the Wasserstein distance.
\end{proof}
A crucial feature of MCMC methods are \emph{law of large numbers} (LLN) results. For a function $f:\R^d\rightarrow \R$ and a Markov chain $(X_k)_{k \in \N}$ we introduce the notation
\[S_n(f)=\sum\limits_{k=0}^n f(X_k)\]
for the Ces\`aro mean of the function values.
Given that the Markov chain $(X_k)_{k \in \N}$ is ergodic and converges to a distribution $\mu_\infty$, we say that LLN holds if $\frac{1}{n}S_n(f)\rightarrow \mu_\infty(f) = \int f\;\d \mu_\infty$ as $n\rightarrow \infty$. Without LLN, in order to approximate statistics of $\mu_\infty$, it is necessary to simulate multiple independent Markov chains for sufficiently many steps and compute averages across all chains. 
In contrast, with LLN it is sufficient to run a single chain and compute running means across consecutive iterates. While this does not directly imply reduced computation time, it is more convenient to implement and memory efficient, as we may successively update the sample average without storing all iterates. The following result shows that the explicit scheme not only satisfies LLN, but also provides a quantification of the convergence that is in line with the rate obtained from a central limit theorem.
\begin{theorem}
\label{prop:LLN_explicit}
    For every $f\in L^1(\pi^\tau)$ and the Markov chain $(X_k^\tau)_{k \in \N}$ initialized at $X_0^\tau=x_0$ for almost any $x_0\in\R^d$ it holds true that
    \begin{equation*}
    \lim\limits_{n\rightarrow \infty}\frac{1}{n}S_n(f)=\pi^\tau(f)\quad\text{a.s.}
    \end{equation*}
    If, in addition, $f$ satisfies $\sup_{x\in\R^d} \left|f(x)-\pi^\tau(f)\right|\;(1+\|x\|^2)^{-\frac{1}{2}}<\infty$ then 
    \[
    \E \left[ \left|\frac{1}{n}S_n(f)-\pi^\tau(f) \right|^2 \right] = \Oc \left(\frac{1}{n} \right).
    \]
\end{theorem}

\begin{remark}
    \label{rem:thm_LLN_EM}
    We refer to~\cite[Theorem 3.1]{latuszynski2013nonasymptotiv} for an explicit formula for the constant in the $\mathcal{O}$-term. We emphasize that the constant in~\Cref{prop:LLN_explicit} depends on properties of the function 
   $f$ and on the step size $\tau$ in a non-explicit manner. As a consequence, the precise bias–variance trade-off with respect to $\tau$, for a fixed number of samples $n$, cannot be quantified within our current analysis. One expects the usual qualitative behavior: decreasing $\tau$
   reduces the discretization bias but deteriorates the constant in the variance estimate. Making this trade-off explicit appears to be beyond the scope of this article and is not addressed here.
\end{remark}

\begin{proof}
    We verify the conditions for the application of Theorem 17.0.1 from \cite{meyn2012markov}: We have already shown that the chain admits a stationary distribution; thus we are left to show that the chain is positive Harris. The function $V(x)=1+\|x\|^2$ is unbounded off petite sets \cite[Section 8.4.2]{meyn2012markov}, as the sublevel sets $C_V(n)=\{y\;|\;V(y)\leq n\}$ are small as in~\Cref{lemma:minor} which implies petiteness \cite[Proposition 5.5.3]{meyn2012markov}. Moreover, $V(x)$ is positive and by \Cref{prop:explicit_scheme_ergodic}, in particular we have $\Delta V(x)\leq 0,\quad x\in C^c$ for a petite set $C$. The Markov chain is therefore Harris \cite[Theorem 9.1.8]{meyn2012markov}. Irreducibility follows directly from the additive Gaussian noise. Altogether the first part of the theorem follows. The additional part is a direct consequence of Theorems 3.1 and 4.2 and Proposition 4.5 in \cite{latuszynski2013nonasymptotiv}. 
\end{proof}
So far we have proven that the explicit scheme is ergodic and that LLN holds. Recall, however, that we are interested in sampling from the distribution $\pi(x)\sim e^{-U(x)}$. Thus, what is left to analyze is the relation between the stationary distribution $\pi^\tau$ of the explicit scheme and the target distribution $\pi$ for which we need the following lemma.
\begin{lemma}
\label{lem:bounded_second_moments_explicit}
If the second moment of $X_0^\tau$ is finite, then the second moments of $(X_k^\tau)_{\tau,k \in \N}$ are uniformly bounded, that is, there exists $\tau_{\text{max}}>0$ such that
\[
\sup_{\tau \in [0,\tau_{\text{max}}]} \sup_{k \in \N} \mathbb{E} \left[ \|X_k^\tau\|^2 \right] < \infty.
\]
\end{lemma}
\begin{proof}
Let $K>0$ be as in \eqref{eq:lyapunov1} such that for $\|x\|\geq K$ we have  $\|x-\tau\theta(x)\|^2\leq\|x\|^2(1-m\tau)$. By the linear growth of $\theta(x)$, for $\|x\|\leq K$ it holds that $\|x-\tau\theta(x)\|\leq K+\tau c(K+1)\eqqcolon M_\tau$. Distinguishing the cases $\|X_k^\tau\|\leq K$ and $\|X_k^\tau\|> K$ we find
\begin{equation*}
    \begin{aligned}
        \|X_k^\tau\|^2 =&
        \|X_{k-1}^\tau - \tau\theta(X_{k-1}^\tau)\|^2 + 2\tau\|W_{k-1}\|^2 + 2\langle X_{k-1}^\tau-\tau\theta(X_{k-1}^\tau),\sqrt{2\tau}W_{k-1}\rangle\\
        \leq& \1_{\{\|X_{k-1}^\tau\|>K\}}(1-m\tau)\|X_{k-1}^\tau\|^2 + \1_{\{\|X_{k-1}^\tau\|\leq K\}}M_\tau^2\\
        &+ 2\tau\|W_{k-1}\|^2 + 2\langle X_{k-1}^\tau-\tau\theta(X_{k-1}^\tau),\sqrt{2\tau}W_{k-1}\rangle.
    \end{aligned}
\end{equation*}
Iterating the case distinction and denoting $R_k = 2\langle X_{k}^\tau-\tau\theta(X_{k}^\tau),\sqrt{2\tau}W_{k}\rangle$ we find
\begin{equation}\label{eq:bounded_expl}
    \begin{aligned}
        \|X_k^\tau\|^2&\leq \1_{\{\|X_n^\tau\|>K\; n=0,1,\dots,k-1\}}(1-m\tau)^k\|X_0^\tau\|^2\\
        &+ \sum\limits_{n=0}^{k-1} \left\{(1-m\tau)^n\left[ 2\tau \|W_{k-1-n}\|^2 + R_{k-1-n}\right]\vphantom{\prod\limits_{\ell=n+1}^{k-1}} + M_\tau^2\1_{\|X_n^\tau\|\leq K}\prod\limits_{\ell=n+1}^{k-1}\1_{\|X_\ell\|> K}\right\}.
    \end{aligned}
\end{equation}
By independence of $X_k$ and $W_k$ for all $k$ we have $\E[R_k]=0$. Moreover, 
\begin{equation*}
    \begin{aligned}
        \E\bigg[\sum_{n=0}^{k-1}\1_{\|X_n^{\tau}\|\leq K}&\prod_{\ell=n+1}^{k-1}\1_{\|X_\ell^{\tau}\|> K}\bigg]\\
        =&\sum_{n=0}^{k-1}\P\big[\text{$\|X_n^{\tau}\|\leq K$, and for $\ell=n+1\dots,{k-1}$: $\|X_\ell^{\tau}\|> K$}\big]\\
        \leq& 1
    \end{aligned}
\end{equation*}
where the last inequality follows since the events are disjoint. Thus, we can take the expected value in~\eqref{eq:bounded_expl} and obtain
\[\E[\|X_k^\tau\|^2]\leq (1-m\tau)^k\E[\|X_0^\tau\|^2] + M_\tau^2 + \sum\limits_{n=0}^{k-1} 2\tau (1-m\tau)^n \leq \E[\|X_0^\tau\|^2] + M_\tau^2 + \frac{2}{m}\]
\end{proof}
\begin{theorem}\label{prop:explicit_convergence_of_stationary}
    The invariant distributions $\pi^\tau$ satisfy $\Wc_2^2(\pi^\tau,\pi)\leq c\tau$ for some $c>0$. In particular, $\Wc_2^2(\pi^\tau,\pi)\rightarrow 0$ as $\tau\rightarrow 0$.
\end{theorem}
\begin{proof}
    In the following we denote $\Ec(\mu)=\int U(x)\;\d \mu(x)$, $\Hc(\mu) = \KL (\mu|\text{Leb})$ with Leb the Lebesgue measure and $\Fc(\mu)=\Ec(\mu)+\Hc(\mu)$. Using strong convexity of $U(x)$ we can compute
    \begin{equation*}
        \begin{aligned}
            \|x-\tau\theta(x)-y\|^2 &= \|x-y\|^2 + 2\tau\langle y-x,\theta(x)\rangle + \tau^2\|\theta(x)\|^2\\
            &\leq \|x-y\|^2 + 2\tau\left(U(y) - U(x) - \frac{m}{2}\|x-y\|^2\right)+ \tau^2\|\theta(x)\|^2\\
            & = (1-m\tau)\|x-y\|^2+\tau^2\|\theta(x)\|^2 + 2\tau\left(U(y) - U(x)\right).
        \end{aligned}
    \end{equation*}
    Plugging in an optimal coupling $(X,Y)$ for $\mu^\tau_{k+1}$ and the target $\pi$ and computing the expectation leads to $2\tau(\Ec(\mu^\tau_{k+1})-\Ec(\pi))\leq (1-m\tau)\Wc_2^2(\mu^\tau_{k+1},\pi) - \Wc_2^2((\mu^\tau_{k+1})^+,\pi) + c\tau^2$ where $c$ is an absolute constant emerging from the linear growth of $\theta(x)$ and the boundedness of the second moments of $(\mu_k^\tau)_{k \in \N}$ and $\mu^+$ denotes the distribution of $X-\tau\theta(X)$ for $X\sim\mu$. As in Proposition 27 from \cite{durmus2019analysis} it follows
    \begin{equation*}
        \begin{aligned}
            2\tau(\Fc(\mu_{k+1}^\tau)-\Fc(\pi))\leq \Wc_2^2((\mu_k^\tau)^+,\pi) -m\tau\Wc_2^2(\mu_{k+1}^\tau,\pi) - \Wc_2^2((\mu_{k+1}^\tau)^+,\pi) + c\tau^2.
        \end{aligned}
    \end{equation*}
    Since $\Fc(\mu_{k+1}^\tau)-\Fc(\pi) = \KL (\mu_{k+1}^\tau|\pi)\geq 0$ \cite[Lemma 1]{durmus2019analysis}, 
    summing over $k$ leads to 
    \begin{equation}\label{eq:convergence_explicit}
        \begin{aligned}
            \frac{m}{n}\sum\limits_{k=0}^{n-1} \Wc_2^2(\mu_{k+1}^\tau,\pi) \leq \frac{\tau^{-1}}{n}\Wc_2^2((\mu_0^\tau)^+,\pi) + c\tau.
        \end{aligned}
    \end{equation}
    We already know that $\mu_{k+1}^\tau\xrightarrow{\Wc_2} \pi^\tau$ and, thus, $\Wc_2^2(\mu_{k+1}^\tau,\pi)\rightarrow \Wc_2^2(\pi^\tau,\pi)$. Taking the limit as $n\rightarrow\infty$ in \eqref{eq:convergence_explicit} then implies $\Wc_2^2(\pi^\tau,\pi)\leq \frac{c}{m}\tau$.
\end{proof}

\begin{remark}
    Contrary to the results from \cite{durmus2019analysis,habring2023subgradient}, in \Cref{prop:explicit_convergence_of_stationary} we do not require $U$ to be Lipschitz continuous (or, equivalently, uniform boundedness of elements in $\partial U$). Moreover, we obtain a convergence result directly for the distributions $(\mu^\tau_k)_{k \in \N}$ whereas \cite{durmus2019analysis} prove convergence for the explicit scheme only for the ergodic means of these distributions.
\end{remark}
\subsection{Semi-implicit Euler-Maruyama scheme}
\label{sec:semi_implicite_scheme}
For this section we will specify a particular class of Potentials $U: \R^d \rightarrow \R$ satisfying \Cref{Ass:Potentials}.
\begin{ass}
\label{Ass:F_G}
    The potential is of the form $U(x)=F(x)+G(x)$ with $F: \R^d \rightarrow \R$ $m$-strongly convex and differentiable with $L$-Lipschitz-continuous gradient and $G: \R^d \rightarrow \R$ convex with subgradients of at most linear growth. Moreover, $G$ admits a minimizer.
\end{ass}
We consider an implicit discretization of the subgradient step with respect to $G$, that is, $X_{k+1}^\tau \in X_k^\tau -\tau (\nabla F(X_k^\tau) + \partial G(X_{k+1}^\tau)) + \sqrt{2\tau}W_k$ or, equivalently,
\begin{equation}\label{eq:semi_implicit_scheme}
    \begin{aligned}
        X_{k+1}^\tau = {\prox}_{\tau G}(X_k^\tau -\tau \nabla F(X_k^\tau)  + \sqrt{2\tau}W_k)
    \end{aligned}
\end{equation}
Our analysis of the semi-implicit scheme, will rely on the following properties of the proximal mapping.
\begin{lemma}
\label{lem:prop_prox}
Let $G: \R^d \rightarrow \R$ be convex and $\tau > 0$. Then, $\prox_{\tau G}:\R^d\rightarrow \R^d$ is $1$-Lipschitz and surjective.
\end{lemma}
\begin{proof}
We provide a short proof of these standard results: Well-definedness of the prox for convex $G$ follows from unique existence of minimizers of~\eqref{eq:defn_prox}.
Regarding Lipschitz-continuity, since $x^* = \prox_{\tau G}(x)$ if and only if $ \tau^{-1}(x-x^*) \in \partial G(x^*)$ and since $\partial G$ is monotone, for $x,y \in \R^d$, it follows that $\tau^{-1}\langle x-x^*-(y-y^*), x^*-y^*\rangle \geq 0 $ implying $\lVert x^* - y^* \rVert \le \lVert x-y\rVert$, which shows that $\prox_{\tau G}$ is $1$-Lipschitz. 
Regarding surjectivity, let $x^*\in\R^d$ be arbitrary and choose $x\in x^*+\tau\partial G(x^*)$. Then, by definition $x^* = \prox_{\tau G}(x)$.
\end{proof}

As in the previous section, we first establish ergodicity and LLN, then convergence to $\pi$.
\begin{theorem}
\label{prop:semi_implicit_fixed_point}
    Let~\Cref{Ass:F_G} hold with $m$ the modulus of strong convexity of $F$ and let $L$ be the Lipschitz constant of $\nabla F$. Then, the proximal-gradient scheme with constant step size $0<\tau<\tfrac{2}{m+L}$ is geometrically ergodic with respect to the Wasserstein-$2$ distance, more precisely, there exists a distribution $\pi^\tau$ on $\R^d$ such that 
    \begin{equation}\label{eq:wasserst_impl}
        \Wc_2(\mu_k^\tau,\pi^\tau)\leq  \left(1-\tau\tfrac{2m L}{m + L} \right)^{\frac{k}{2}} \Wc_2(\mu_0^\tau,\pi^\tau).
    \end{equation}
\end{theorem}
\begin{remark}\
\begin{itemize}
    \item The constant $\Wc_2(\mu_0^\tau,\pi^\tau)$ can be bounded in terms of $L$, $m$ and $d$ by initializing $\mu^\tau_0$ as a Dirac at the minimizer of $U$.
    \item In contrast to the explicit discretization, the semi-implicit scheme yields an improved convergence rate of order $e^{-k\tau\tfrac{m L}{m + L}}$ as a function of $\tau$. In addition, the constant on the right-hand side of~\eqref{eq:wasserst_impl} is solely the initial error without an additional factor.
\end{itemize}   
\end{remark}
\begin{proof}
    Let $\mu,\nu\in \Pc_2(\R^d)$ be two probability distributions and the random variables $X\sim \mu$, $Y\sim\nu$. We define the random variables after one iteration
    \begin{equation*}
        \begin{aligned}
            X^+ &= {\prox}_{\tau G}(X -\tau \nabla F(X)  + \sqrt{2\tau}W_k)\\
            Y^+ &= {\prox}_{\tau G}(Y -\tau \nabla F(Y)  + \sqrt{2\tau}W_k)
        \end{aligned}
    \end{equation*}
    where we emphasize that the steps are coupled, that is, for both transitions we use the same $W_k$. By 1-Lipschitz continuity of the prox, $m$-strong convexity of $F$, and $L$-Lipschitz continuity of $\nabla F$, using \cite[Theorem 2.1.12]{nesterov2013introductory}, we find
    \begin{equation}\label{eq:contraction_implicit}
        \begin{aligned}
            \Wc^2_2(\mu^+,\nu^+) & \leq\E\left[ \|X^+-Y^+\|^2\right] \\ 
            & \leq \E\left[\|X-Y -\tau (\nabla F(X) -\nabla F(Y))\|^2\right]\\
            &= \E\left[\|X-Y\|^2 + \tau^2 \|\nabla F(X)-\nabla F(Y)\|^2 - 2 \tau \langle X-Y,\nabla F(X)-\nabla F(Y)\rangle\right]\\
            &\leq \E\left[\|X-Y\|^2 + \tau^2 \|\nabla F(X)-\nabla F(Y)\|^2 - \tau\frac{2m L}{m + L}\|X-Y\|^2  \right.\\
            & \left.\qquad - \tau\frac{2}{m+L}\|\nabla F(X)-\nabla F(Y)\|^2\right] \\
            &\leq \left (1-\tau\frac{2m L}{m + L} \right) \E\left[\|X-Y\|^2\right] + \tau\left (\tau - \frac{2}{m+L}\right) \E\left[\|\nabla F(X)-\nabla F(Y)\|^2\right],
        \end{aligned}
    \end{equation}
    where the second term is non-positive by the step size constraint. Taking the infimum over all couplings $(X,Y)$ with $X\sim \mu$ and $Y\sim\nu$ on the right-hand side yields a contraction and the rest follows from Banach's fixed point theorem.
\end{proof}
\begin{theorem}
    The semi-implicit scheme with step size $0<\tau<\tfrac{2}{m+L}$ satisfies LLN. That is, for every $f\in L^1(\pi^\tau)$ and for almost all $x_0 \in \R^d$, if the Markov chain $(X_k^\tau)_{k \in \N}$ is initialized at $X_0^\tau=x_0$, it holds true that
    \[
    \lim\limits_{n\rightarrow \infty}\frac{1}{n}S_n(f)=\pi^\tau(f)\quad\text{a.s.}
    \]
    If, in addition, $f$ satisfies $\sup_{x\in\R^d} \left|f(x)-\pi^\tau(f)\right|\;(1+\|x\|^2)^{-\frac{1}{2}}<\infty$ then 
    \[
    \E \left[ \left|\frac{1}{n}S_n(f)-\pi^\tau(f) \right|^2 \right] = \Oc \left(\frac{1}{n} \right).
    \]
\end{theorem}
\begin{proof}
    The proof follows similar steps as in the explicit case, that is, we have to verify the conditions of Theorem 17.0.1 in \cite{meyn2012markov}.
    
    \noindent\textbf{Step 1:} The scheme is irreducible with respect to the Lebesgue measure: Let $A\in\Bc(\R^d)$, 
    \begin{equation*}
        \begin{aligned}
            \P\left[ X_1^\tau\in A \;\middle|\; X_0^\tau=x \right]=&\P\left[ {\prox}_{\tau G}(X_0^\tau - \tau\nabla F(X_0^\tau) + \sqrt{2\tau}W_0)\in A \;\middle|\; X_0^\tau=x \right]\\
            =&\P\left[ (I + \tau\partial G)^{-1}(X_0^\tau - \tau\nabla F(X_0^\tau) + \sqrt{2\tau}W_0)\in A \;\middle|\; X_0^\tau=x \right]\\
            = &\P\left[ X_0^\tau - \tau\nabla F(X_0^\tau) + \sqrt{2\tau}W_0\in (I+\tau\partial G)(A) \;\middle|\; X_0^\tau=x \right]\\
            = &\P\left[ \sqrt{2\tau} W_0 \in (I+\tau\partial G)(A) +\tau \nabla F(x)-x \right].
        \end{aligned}
    \end{equation*}
    Assume that the right-hand side is zero, that is, $B\coloneqq(I+\tau\partial G)(A)$ has Lebesgue measure zero. By~\Cref{lem:prop_prox} ${\prox}_{\tau G}$ is surjective and thus $A\subset {\prox}_{\tau G}({\prox}_{\tau G}^{-1}(A))={\prox}_{\tau G}(B)$. As the proximal mapping is $1$-Lipschitz, it follows that ${\prox}_{\tau G}(B)$ is a Lebesgue zero set \cite[Lemma 7.25]{rudin1987real} and, thus, $A$ is as well. Next we show Harris recurrence for which we need a minorization and a drift condition, cf.~Theorem 9.1.8 in \cite{meyn2012markov}.
    
    \noindent\textbf{Step 2:} Minorization condition: Let $C=\{x\in\R^d\;|\;\|x\|^2\leq R\}$ and $x\in C$. For $A\in\Bc(\R^d)$,
    \begin{equation}
        \begin{aligned}
            \P[X_1^\tau\in A\;|\; X_0^\tau=x]
            =&\P[x-\tau\nabla F(x) + \sqrt{2\tau}W_0\in (\id+ \tau\partial G)(A)]\\
            =& \int\limits_{(\id+ \tau\partial G)(A)} \frac{1}{(2\pi 2\tau)^{d/2}}e^{-\frac{\|x-\tau\nabla F(x) - z\|^2}{4\tau}}\d z\\
            \geq& \int\limits_{(\id+ \tau\partial G)(A)} \frac{1}{(2\pi 2\tau)^{d/2}}e^{-\frac{2\|x-\tau\nabla F(x)\|^2 + 2\|z\|^2}{4\tau}}\d z\\
            \geq& \frac{1}{2^{d/2}}e^{-\frac{\|x-\tau\nabla F(x)\|^2}{2\tau}}\int\limits_{(\id+ \tau\partial G)(A)} \frac{1}{(2\pi \tau)^{d/2}}e^{-\frac{\|z\|^2}{2\tau}}\d z\\
            \geq& \underbrace{\inf_{x\in C}\frac{1}{2^{d/2}}e^{-\frac{\|x-\tau\nabla F(x)\|^2}{2\tau}}}_{\eqqcolon \alpha(\tau)}\int\limits_{(\id+ \tau\partial G)(A)} \frac{1}{(2\pi \tau)^{d/2}}e^{-\frac{\|z\|^2}{2\tau}}\d z\\
            =& \alpha(\tau)\nu_\tau(A)
        \end{aligned}
    \end{equation}
    with $\alpha(\tau)>0$ and $\nu_\tau = \law(\prox_{\tau G}(Z))$ with $Z\sim\Nc(0,\tau)$ yielding the minorization condition.
    
    \noindent\textbf{Step 3:} Drift condition:
    Let us consider $V(x)=1+\|x\|^2$ which is positive and unbounded off petite sets. Note that, since we assume (without loss of generality) that $0$ is a minimizer of $U$ we have $0 = {\prox}_{\tau G}(-\tau \nabla F(0))$.
    Using $m$-strong convexity of $F$, non-expansiveness of the prox, and the same estimate as in~\eqref{eq:contraction_implicit} we find
    \begin{equation}\label{eq:implicite_LLN}
        \begin{aligned}
            \|X_1^\tau\|^2&
            = \|{\prox}_{\tau G}(X_0^\tau - \tau\nabla F(X_0^\tau)+\sqrt{2\tau}W_0) - {\prox}_{\tau G}(-\tau \nabla F(0))\|^2\\
            \leq& \|X_0^\tau - \tau\nabla F(X_0^\tau)+\sqrt{2\tau}W_0 - (0-\tau \nabla F(0))\|^2\\
            \leq& \left(1-\tau\frac{2mL}{m+L} \right)\|X_0^\tau\|^2 + 2\tau\|W_0\|^2
            + 2\langle X_0^\tau - \tau\nabla F(X_0^\tau) + \tau\nabla F(0)), \sqrt{2\tau}W_0\rangle.
        \end{aligned}
    \end{equation}
    Consequently, for $\|x\|\geq K$ with appropriate $K$
    \begin{equation*}
        \begin{aligned}
            \Delta V(x) \coloneqq \E\left[ V(X_1^\tau)\;\middle|\; X_0^\tau=x \right] - V(x)
            & \leq \left(1-\tau\frac{2mL}{m+L} \right)\|x\|^2 + 2\tau - \|x\|^2\\
            & \leq -\tau\left(\frac{2mL}{m+L}K^2 -2 \right)
        \end{aligned}
    \end{equation*}
    which is negative if $K$ is sufficiently large. Thus, with $C=\overline{B_K(0)}$, $\Delta V(x)\leq 0$ for $x\in C^c$. By step 2, the set $C$ is small and therefore petite rendering the chain Harris \cite[Theorem 9.1.8]{meyn2012markov}. We have already shown that the chain is ergodic and irreducible, rendering it \emph{positive} Harris. Thus, LLN is satisfied \cite[Theorem 17.0.1]{meyn2012markov}. As for the explicit scheme the additional statement is a direct consequence of Theorems 3.1 and 4.2 and Proposition 4.5 by \cite{latuszynski2013nonasymptotiv}. 
\end{proof}
\begin{lemma}
\label{lem:bounded_second_moments}
If the second moment of $\mu_0$ is finite, then the second moments of $(\mu_k^\tau)_{\tau,k \in \N}$ are uniformly bounded, that is, for any $\tau_{\text{max}}<\frac{2}{m+L}$ it holds
\[
\sup_{\tau \in [0,\tau_{\text{max}}]} \sup_{k \in \N} \mathbb{E} \left[ \|X_k^\tau\|^2 \right] < \infty.
\]
\end{lemma}
\begin{proof}
From~\eqref{eq:implicite_LLN} it follows
    \begin{equation*}
        \begin{aligned}
            \Ex[\|X_{k+1}^\tau\|^2]
            &\leq \Ex \left[\left(1-\tau\frac{2 m L}{m+L}\right)\|X_k^\tau\|^2+2\tau\|W_k\|^2\right] \\
            &= \left(1-\tau\frac{2 m L}{m+L}\right)\Ex[\|X_{k}^\tau\|^2] + 2\tau.
        \end{aligned}
    \end{equation*}
    Solving the recursion yields the desired result.
\end{proof}
\begin{remark}
    As a consequence also the family $(Y_k^\tau)_{k \in \N,\tau}$ defined by $Y_{k+1}^\tau = X_{k+1}^\tau - \tau\nabla F(X_{k+1}^\tau)+\sqrt{2\tau} W_{k+1}$ has uniformly bounded second moments in the sense of \Cref{lem:bounded_second_moments}.
\end{remark}
\begin{theorem}
\label{prop:implicit_final_convergence}
The stationary distribution $\pi^\tau$ of the semi-implicit scheme satisfies $\Wc^2_2(\pi^\tau,\pi)\leq c\tau$ for some $c>0$.
\end{theorem}
\begin{proof}
    Under the additional assumption that $G$ is Lipschitz, \cite{durmus2019analysis} show that with step size $\tau>0$, for $k$ large enough, it holds true that $\Wc^2_2(\mu_k^\tau,\pi)\leq C\tau$ for some constant $C>0$ \cite[Theorem 21]{durmus2019analysis}. We will show that this result remains true under the assumptions of this paper. As a consequence, since we have already proven that the sequence $(\mu_k^\tau)_{k \in \N}$ converges to some stationary distribution $\pi^\tau$ with respect to $\Wc_2$ as $k \rightarrow \infty$, it follows $\Wc^2_2(\pi^\tau,\pi)\leq C\tau$. A closer investigation of Lemmas 29, 30 from \cite{durmus2019analysis} reveals that, in order to ensure an approximation in Wasserstein-$2$ distance, it is sufficient to guarantee that for $Y_{k+1}^\tau = X_{k+1}^\tau - \tau\nabla F(X_{k+1}^\tau)+\sqrt{2\tau} W_{k+1}$ we have
    \[
    \sup_{k \in \N} \Ex[|G(Y_k^\tau)-G({\prox}_{\tau G}(Y_k^\tau))|] = \Oc(\tau) \quad \text{as } \tau\rightarrow 0.
    \]
    We denote $f_\tau(y)\coloneqq G(y)-G({\prox}_{\tau G}(y))$ and we note that $f_{\tau} (y) \ge 0$ for all $y \in \R^d$ by the definition of $\text{prox}_{\tau G}$. 
    By~\Cref{Ass:F_G} $G$ admits a miminizer which we now denote as $x^*$. By the definition of the prox it follows
    \begin{equation*}
        \frac{1}{2\tau}\|y-{\prox}_{\tau G}(y)\|^2 + G({\prox}_{\tau G}(y))\leq \frac{1}{2\tau}\|y-x^*\|^2 + G(x^*)
    \end{equation*}
    and, since $G({\prox}_{\tau G}(y))\geq G(x^*)$, $\|y-{\prox}_{\tau G}(y)\| \leq \|y-x^*\|$ implying $\|{\prox}_{\tau G}(y)\|\leq c(\|y\|+1)$ for some $c>0$.
    By the linear growth of the subdifferential of $G$, it follows $\|\partial G ({\prox}_{\tau G}(y))\|\leq c (1 +\|y\|)$.
    Noting that $\frac{1}{\tau}({\prox}_{\tau G}(y) - y)\in\partial G({\prox}_{\tau G}(y))$ and putting everything together, we can estimate
    \begin{equation*}
        \begin{aligned}
            0 \le f_\tau(y) =  G(y)-G({\prox}_{\tau G}(y)) 
            \leq |\langle\partial G(y), {\prox}_{\tau G}(y) - y\rangle|
            &\leq c (\|y\|+1) \|{\prox}_{\tau G}(y) - y\|\\
            &\leq \tilde{c} (\|y\|^2+ \|y\| +1) \tau,
        \end{aligned}
    \end{equation*}
    for some $\tilde{c}>0$. Thus, $\Ex[f_\tau(Y_k^\tau)] \leq \Ex[\tilde{c} (\|Y_k^\tau\|^2+ \|Y_k^\tau\| +1) \tau]$ and hence by \Cref{lem:bounded_second_moments}
    \[
    \sup_{k \in \N} \Ex[f_\tau(Y_k^\tau)] \leq \sup_{k \in \N} \left[  \Ex[\tilde{c} (\|Y_k^\tau\|^2+ \|Y_k^\tau\| +1) \tau] \right] = \mathcal{O} \left( \tau \right).
    \]
\end{proof}

\section{Numerical Experiments}\label{sec:experiments}
In this section we show several numerical experiments confirming the proven convergence results and verifying the practical relevance of the investigated algorithms. We will refer to \eqref{eq:explicit_scheme} as the \emph{explicit scheme} and to \eqref{eq:semi_implicit_scheme} as the \emph{implicit scheme}. In all examples we consider potentials factoring as $U(x)=F(x)+G(x)$. The code to reproduce the results is publicly available \cite{habring2024subgrad_git}. Additional implementation details can be found in the supplemental material, \Cref{S-sec:implementation_details}
\subsection{Experiments in 2D}\label{sec_2d_experiments}
For the first set of experiments we consider distributions on $\R^2$. The functional $F$ will be fixed as $F(x) = \frac{1}{2}\|x-x_0\|^2$ with $x_0 \in (0,1)$. For $G(x)$ we consider two different settings as will be explained in the following two subsections. The purpose of these low-dimensional experiments is to confirm the convergence guarantees in the most general setting. For practical examples we refer the reader to \Cref{section:experiments_imaging} where we apply the algorithms to image denoising and deconvolution. To evaluate the methods in the low-dimensional case we will sample multiple independent Markov chains in parallel and estimate the Wasserstein-$2$ distance (using POT~\cite{flamary2021pot,flamary2024pot}) and total variation distances between the empirical distribution induced by the sample and a histogram approximating the target density $\pi$. For details regarding the computation of these metrics we refer to~\Cref{S-sec:computation_of_metrics} in the supplement. Moreover, in~\Cref{S-sec:physical_time} in the supplement, in addition to the presented convergence plots we show plots of error over physical time instead of the number of iterations supporting the approximation of an underlying continuous time process.
\subsubsection{Without linear operator: $U(x)=F(x)+G(x)$}
\label{sec:without_operator}
As a first choice for $G$ we consider $G(x) = \theta\left(|x_1|_{*} + |x_2|_{*}\right)$, with the \emph{mixed norm} $|t|_{*} = t$ if $t\geq 0$ and $|t|_{*}=\frac{2}{3}|t|^\frac{3}{2}$ if $t< 0$ (see \Cref{fig:mixed_norm}).
\begin{figure}
    \centering
    \includegraphics[width=0.5\linewidth]{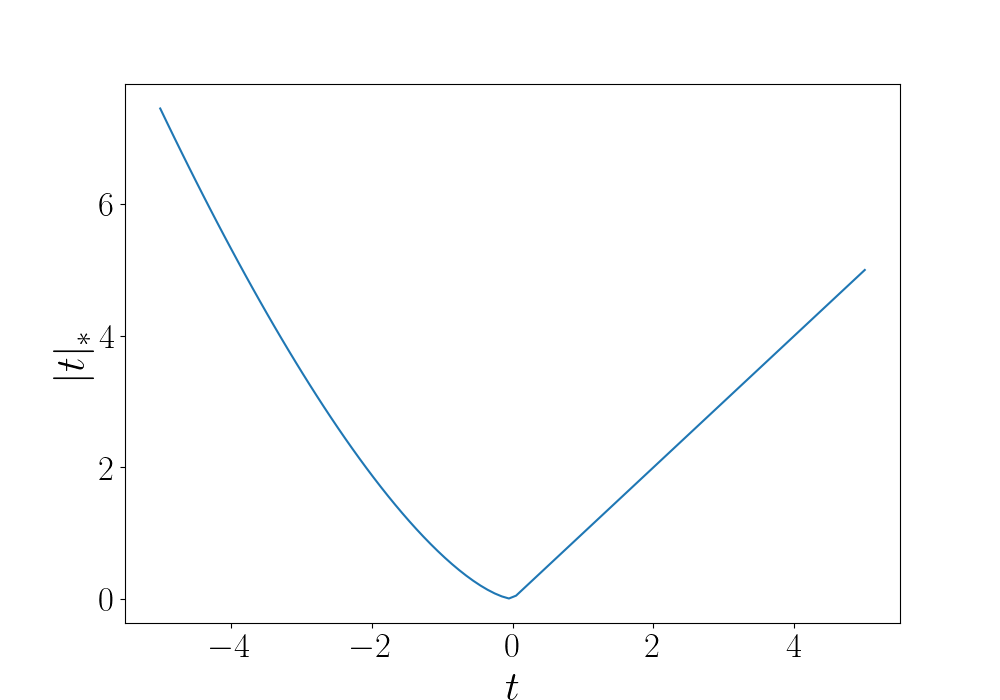}
    \caption{Mixed norm $|t|_*$ used in the 2D examples.}
    \label{fig:mixed_norm}
\end{figure}
We set the regularization parameter $\theta = 5$ in this example. Note that such $G(x)$ is not differentiable in $x=0$ and for $x<0$ it is neither Lipschitz nor gradient Lipschitz. Thus, this particular choice incorporates all the possible non-regularities which can be handled with the proposed sampling methods and for which previously direct sampling (without prior smoothing of the potential) had not been possible. For comparison we will also add the results obtained by sampling using MYULA \cite{durmus2022proximal}, that is,
\begin{equation}
\label{eq:myula}
    X_{k+1}^M = \left (1-\frac{\tau}{\lambda} \right) X_k^M - \tau \nabla F(X_k^M) + \frac{\tau}{\lambda}{\prox}_{\lambda G}(X_k^M) + \sqrt{2\tau}W_{k+1}. \tag{MYULA}
\end{equation}
MYULA corresponds to sampling from a smoothed potential $U^\lambda = F+G^\lambda$ where $G$ is replaced by its Moreau envelope $G^\lambda$. The parameter $\lambda$ for the Moreau envelope needs to satisfy $\tau\leq\frac{\lambda}{\lambda L +1}$ \cite[Section 3.3]{durmus2022proximal} and we set it as $\lambda = \frac{\tau}{1-L\tau}$ to obtain minimal bias. In \Cref{fig:without_K_convergence} we show the Wasserstein-2 and---since the convergence of MYULA is proven in TV \cite[Theorem 2]{durmus2022proximal}---the total variation distance between the target density and the iterates of the proposed algorithms and MYULA for different step sizes. For all methods we can observe the desired exponential convergence and a saturation at a non-zero bias which decreases with $\tau$. We also present the computational complexity of the different methods, for which we do not find any significant differences in this experiment, see the left part of \Cref{table:computation_times_2d}.
\begin{figure}[h]
    \centering
    \includegraphics[width=0.75\linewidth]{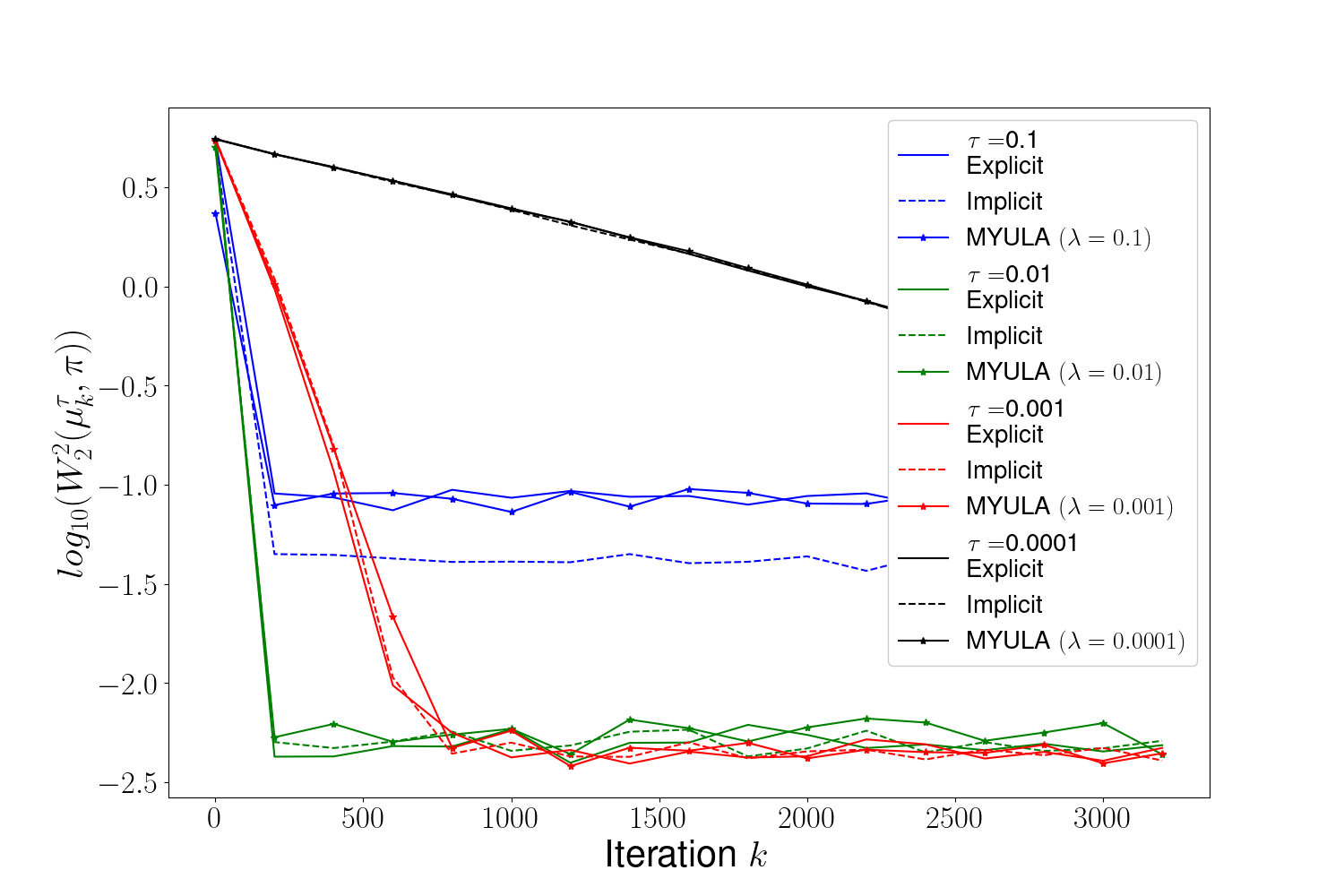}
    \includegraphics[width=0.75\linewidth]{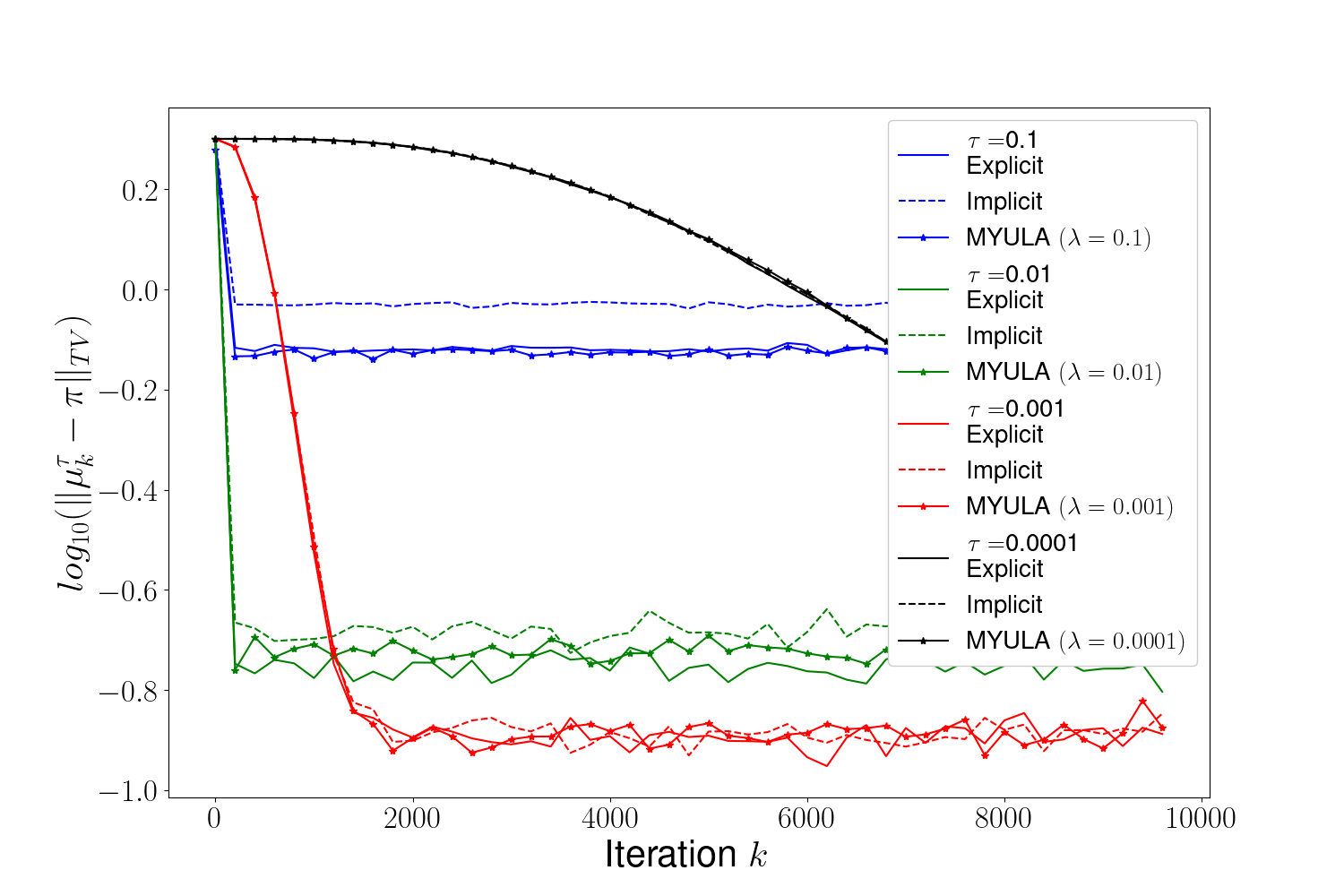}
    \caption{Potential $U(x)=F(x)+G(x)$. On the top we show the Wasserstein-$2$ and on the bottom the total variation distance between samples and target $\pi$ for different step sizes and methods.}
    \label{fig:without_K_convergence}
\end{figure}

\begin{table}[h]
\centering
\begin{tabular}{c|ccc|cc}
 & \multicolumn{3}{c|}{$U(x)=F(x)+G(x)$} & \multicolumn{2}{c}{$U(x)=F(x)+G(Kx)$}\\
\midrule
$\tau$ & MYULA & Explicit& Implicit & MYULA & Explicit\\
\midrule
$10^{-1}$ & 0.61&   0.52&    0.58&  4.9&    0.53\\
$10^{-2}$ & 0.60&   0.51&    0.59&  2.3&    0.54\\
$10^{-3}$ & 0.60&   0.50&    0.59&  15.8&    0.53\\
$10^{-4}$ & 0.60&   0.50&    0.59&  19.3&    0.53
\end{tabular}
\vspace*{0.2cm}
\caption{2D sampling with and without a linear operator $K$. Average computation times in seconds for 1000 iterations for different methods and step sizes $\tau$. We show the times for computing 1e4 parallel Markov chains simultaneously. The computation times in the first column refer to~\Cref{sec:without_operator}, wheres the second column refers to~\Cref{sec:with_linear_operator}. Simulations were performed on the CPU using 24 AMD Ryzen 9 3900X 12-Core Processors.}
\label{table:computation_times_2d}
\end{table}

\subsubsection{With linear operator: $U(x)=F(x)+G(Kx)$}
\label{sec:with_linear_operator}
For this experiment we again use the function $G$ from \Cref{sec:without_operator}, however, we include a non-trivial linear operator within $G$ and consider the potential $U(x) = F(x) + G(Kx)$ with $Kx = x_2-x_1$. Note that in this case the proximal mapping of $G\circ K$ is not explicit. Therefore, we omit the experiments with the semi-implicit variant of the proposed method and for MYULA we compute the prox using the primal-dual algorithm \cite{primal_dual_algo}. The subroutine is terminated when the difference between consecutive iterates is less than 1e-4 in the maximum norm. We show the same convergence plots as in the previous section in \Cref{fig:with_K_convergence} confirming the theoretical results for the Wasserstein-$2$ and the total variation convergence. In this experiment a clear bias reduction of the proposed method compared to MYULA can be observed and is most likely due to the Moreau-Yoshida approximation of $G$. The computation times for this experiment can be found in \Cref{table:computation_times_2d} on the right side. Due to the iterative computation of the prox, we find a significantly greater computational effort for MYULA. Moreover, the computational effort in this setting becomes dependent on the parameter $\lambda$ as it affects the convergence of the primal-dual subroutine.
\begin{figure}[h]
    \centering
    \includegraphics[width=0.75\linewidth]{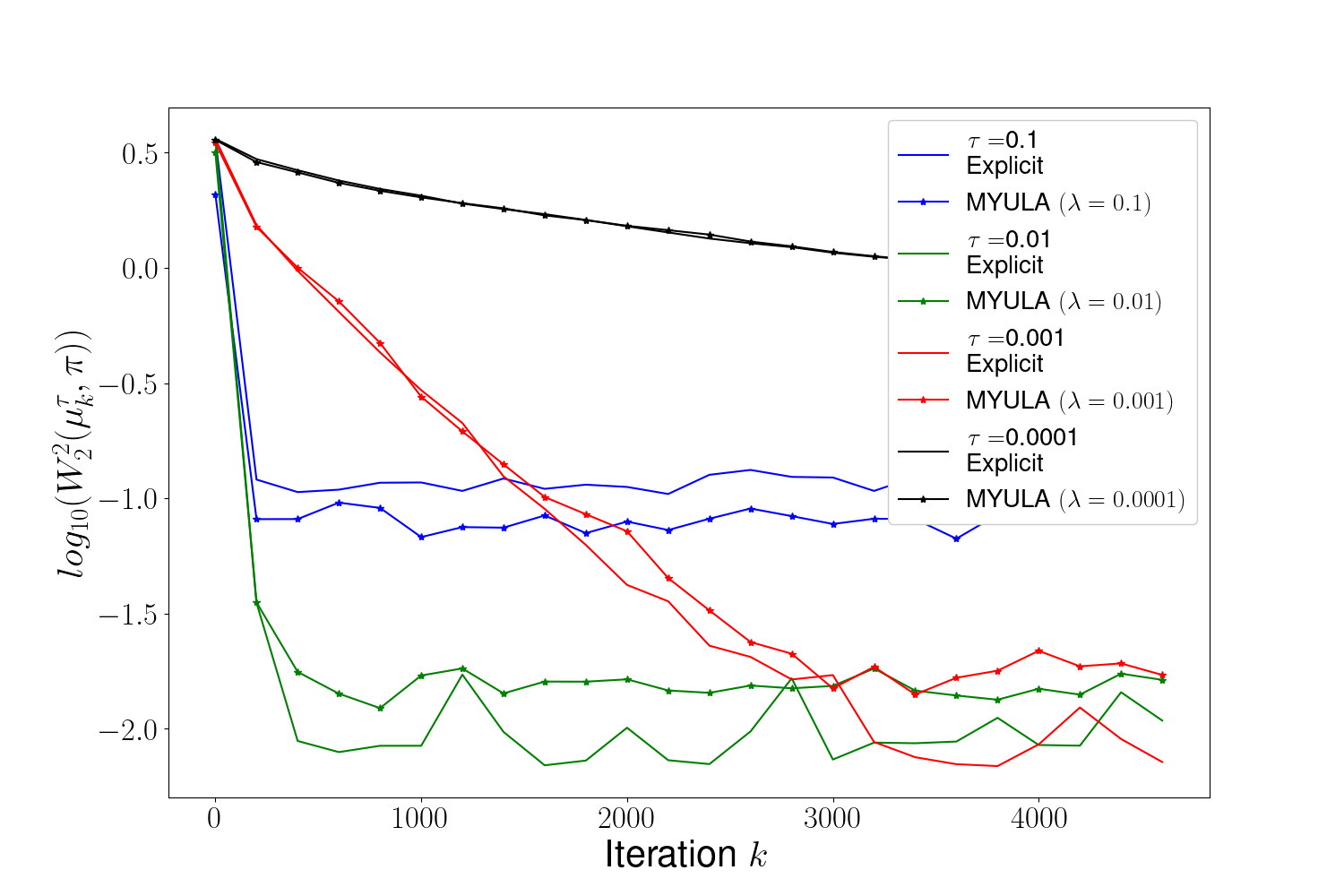}\\
    \includegraphics[width=0.75\linewidth]{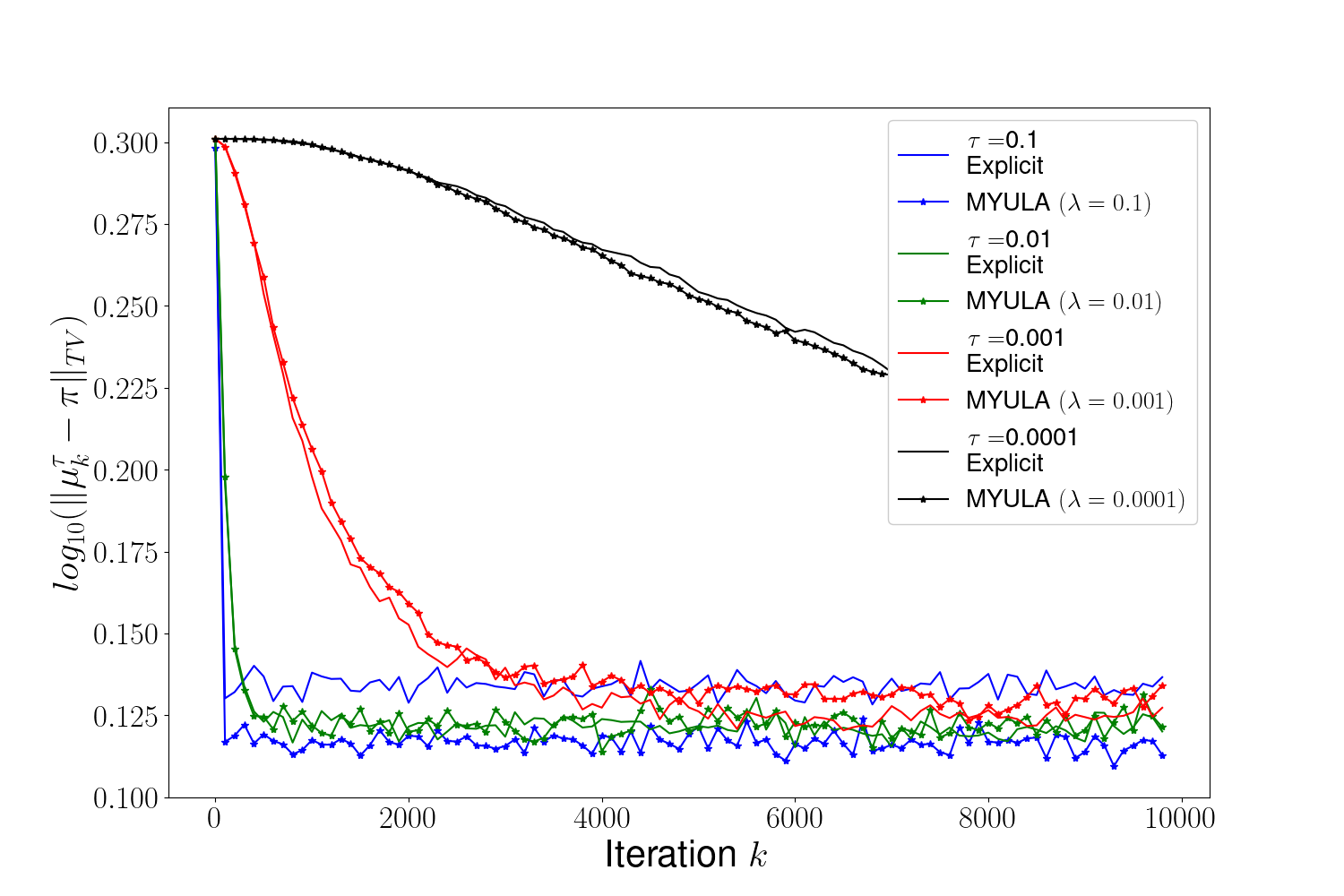}
    \caption{Potential $U(x)=F(x)+G(Kx)$ with linear operator in $G$. On the left we show the Wasserstein-$2$ and on the right the total variation distance between samples and target $\pi$ for different step sizes and methods.}
    \label{fig:with_K_convergence}
\end{figure}

\subsection{Imaging Examples}\label{section:experiments_imaging}
In this section we show two applications in the context of inverse imaging problems, namely image denoising and deconvolution. In both cases, we define $G$ as the total variation functional \cite{rudin1992tv_mh}, $G(x) = \theta \TV(x)$ with $\theta>0$ a regularization parameter and $\TV(x)$ the anisotropic total variation, that is, $\TV(x) = \|Kx\|_{1,1}$ with $\|p\|_{1,1} = \sum_{i=1}^N\sum_{j=1}^M|p_{i,j}^1| + |p_{i,j}^2|$ for $p\in\R^{N\times M\times 2}$ and the finite differences operator $K:\R^{N\times M}\rightarrow \R^{N\times M\times 2}$ defined as
\begin{equation*}
    \begin{aligned}
        (Kx)^1_{i,j} = \begin{cases}
            x_{i+1,j}-x_{i,j}\quad i<N\\
            0\quad \text{otherwise,}
        \end{cases}\quad
        (Kx)^2_{i,j} = \begin{cases}
            x_{i,j+1}-x_{i,j}\quad j<M\\
            0\quad \text{otherwise.}
        \end{cases}
    \end{aligned}
\end{equation*}
While the presented methods would also work for the favorable isotropic TV functional, we will also compare the results to the ones obtained using Belief Propagation (BP) \cite{tapfre03,knobelreiter2020belief,narnhofer2022posterior} which is only applicable in the anisotropic case. 
\subsubsection{Image denoising}
In the case of image denoising the potential reads as $U(x) = \frac{1}{2\sigma^2}\|x-y\|^2 + \theta\TV(x)$. For the data $y$ we use the ground truth image corrupted with additive Gaussian noise with standard deviation $\sigma = 0.05$. The parameter $\theta$ is set to $\theta = 30$. To reduce the computational effort we sample only a single Markov chain. We let the chain run for a burnin phase of 5e5 iterations and afterwards use the following  5e5 samples to approximate expected value and variance using running means which is possible as the chains satisfy the LLN (\Cref{prop:LLN_explicit}). As a \emph{ground truth} to compare to, we employ the BP algorithm which---despite not being exact---yields highly accurate estimates of the marginal distribution for each pixel and a discretized gray-scale space. In \Cref{fig:image_denoising_results}, for visual comparison, we show the final estimates of the expected values for all used methods. While in such high-dimensional examples, computing metrics of the entire distributions is far too computationally expensive, as a proxy we compare the expected values and variances of the explicit scheme and MYULA each to the BP results, which is illustrated in \Cref{fig:image_denoising_convergence}. Note that these expected values and variances are again elements of $\R^d$ with $d$ the number of image pixels and the errors shown are computed as the $L^2$ difference with respect to all pixels between estimate and target. Again, MYULA exhibits an increased bias compared to the proposed method. Lastly, we report the computation times in \Cref{table:computation_times_imaging}. Here, we find a significant increase in computational effort with MYULA caused by the iterative computation of the proximal mapping. However, it should be noted that this could be reduced by employing a rougher approximation of the prox, that is, fewer iterations in the primal-dual subroutine.
\begin{figure}[h]
    \centering
    \begin{subfigure}{\textwidth}
    \centering
    \includegraphics[height=2cm,valign=t]{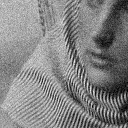}%
    \includegraphics[height=2cm,valign=t]{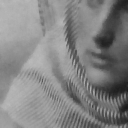}%
    \includegraphics[height=2cm,valign=t]{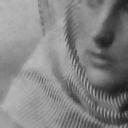}%
    \includegraphics[height=2cm,valign=t]{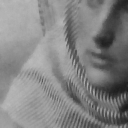}%
    \includegraphics[height=2cm,valign=t]{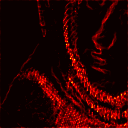}%
    \includegraphics[height=2cm,valign=t]{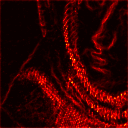}%
    \includegraphics[height=1.9cm,trim={2.72cm 1.8cm 0 1.5cm},valign=t]{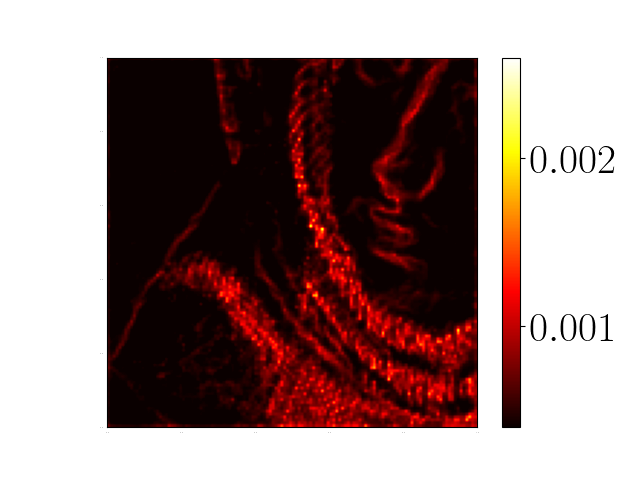}
    \end{subfigure}\\
    \begin{subfigure}{\textwidth}
    \centering
    \includegraphics[height=2cm]{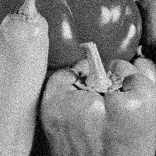}%
    \includegraphics[width=2cm]{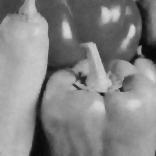}%
    \includegraphics[height=2cm]{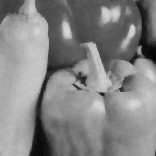}%
    \includegraphics[width=2cm]{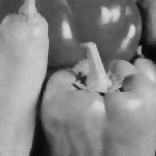}%
    \includegraphics[width=2cm]{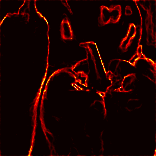}%
    \includegraphics[height=2cm]{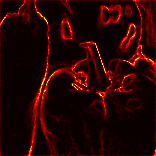}%
    \includegraphics[width=2.88cm,trim={3cm 1.5cm 0 1.5cm},clip]{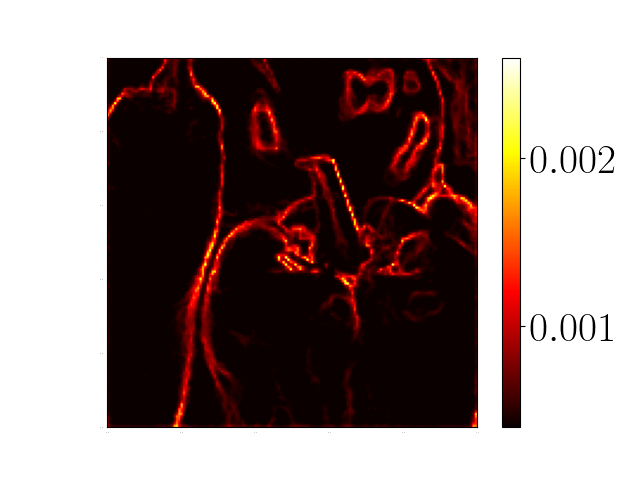}
    \end{subfigure}
    \caption{Denoising: estimated expected values and variances. From left to right: Corrupted image $y$, expected values computed with BP, MYULA, and the proposed explicit scheme. Then variances computed with BP, MYULA, and the proposed explicit scheme. For MYULA and the proposed method we compute the statistics using the following 5e5 iterates after a burnin phase of 5e5 iterations.}
    \label{fig:image_denoising_results}
\end{figure}

\begin{figure}[h]
    \centering
    \includegraphics[width=0.49\textwidth]{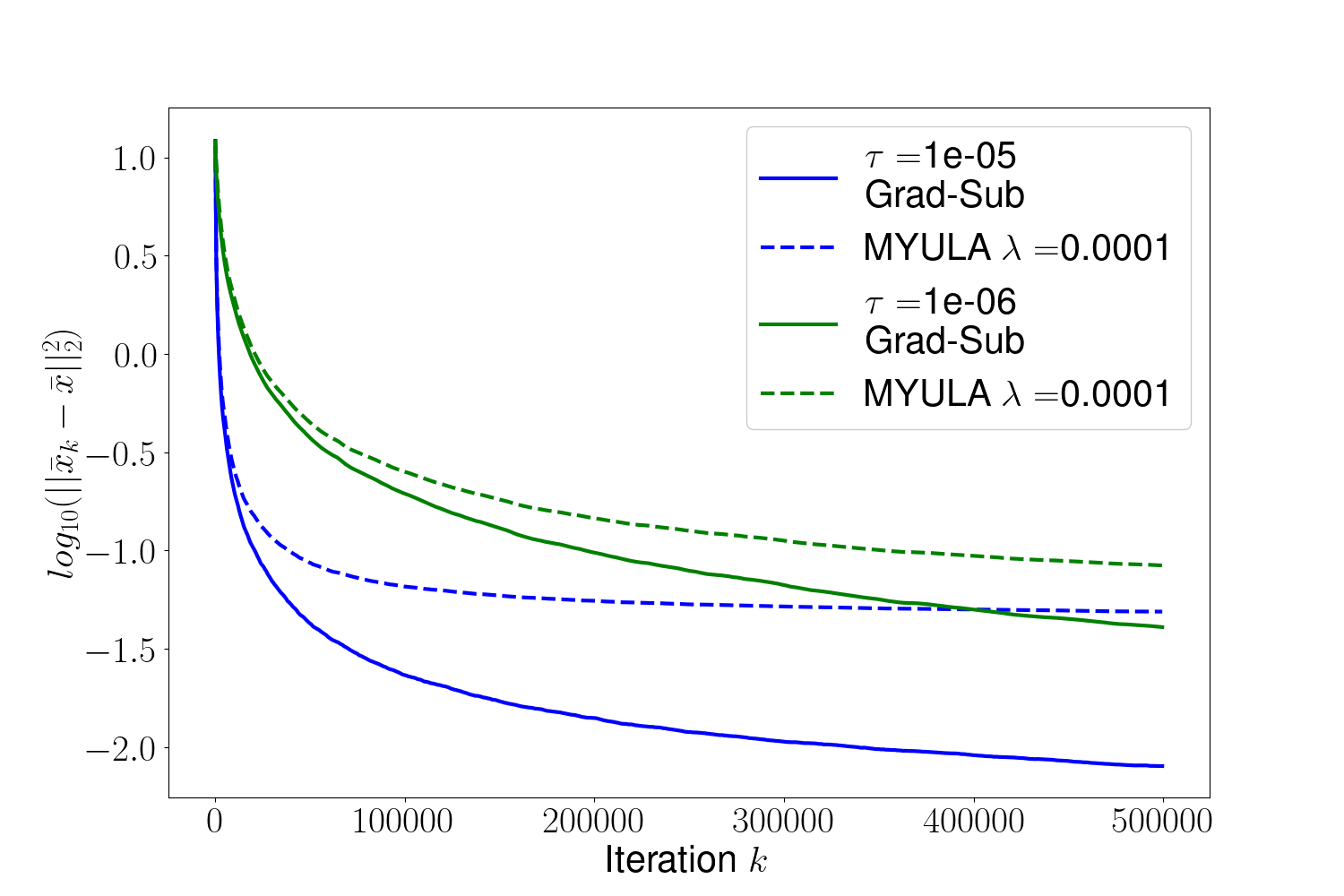}
    \includegraphics[width=0.49\textwidth]{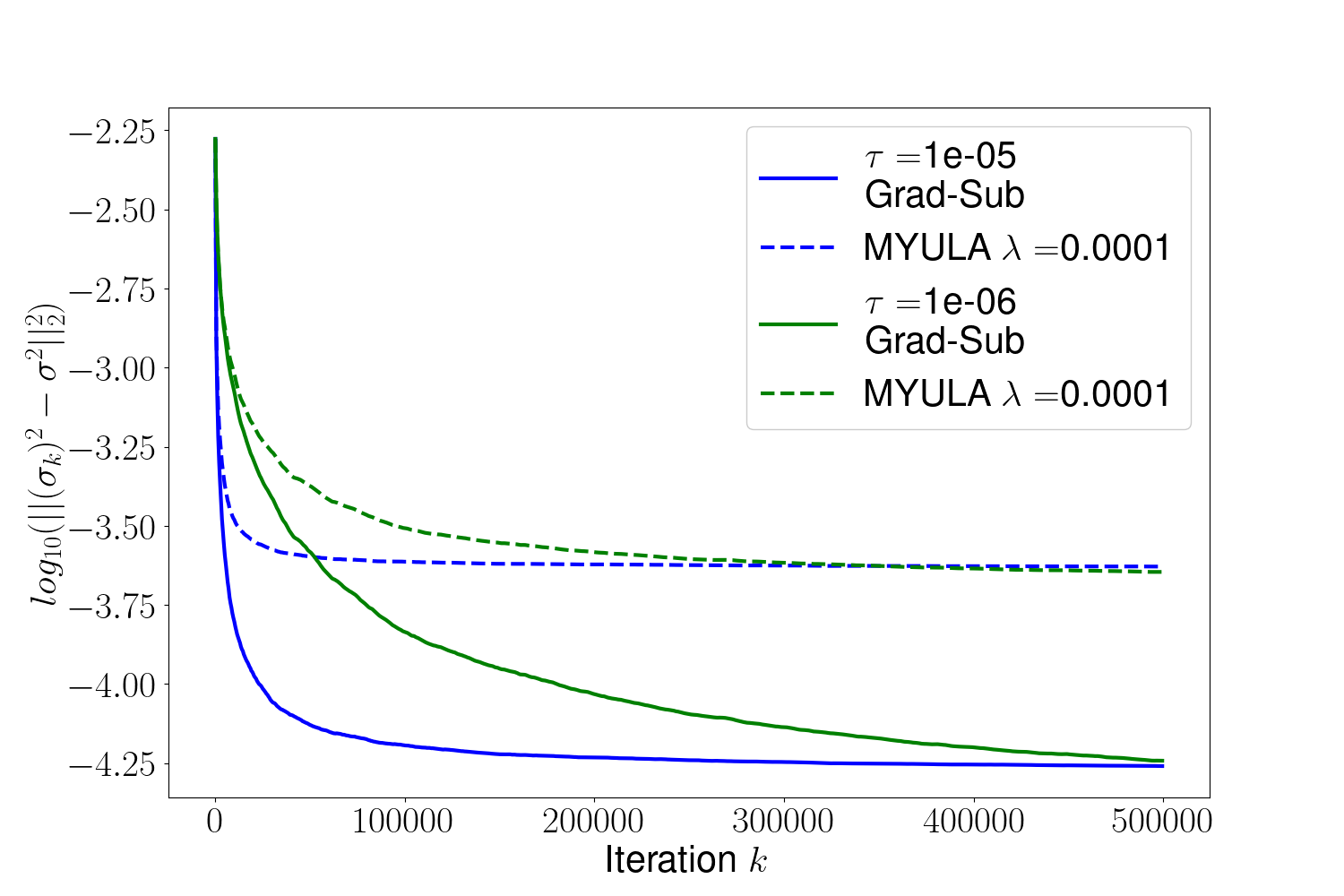}
    \caption{Denoising: $L_2$-error of estimated expected value (left) and variance (right) of the proposed explicit scheme and MYULA each compared to BP results for the peppers image. We use a burnin phase of 5e5. The symbols $\bar{x}_k,\sigma_k$ denote the emprical expected value and variance using $k$ successive iterates, $\bar{x},\sigma$ the estimates from BP.}
    \label{fig:image_denoising_convergence}
\end{figure}
\begin{table}[h]
\centering
\begin{tabular}{c|cc|cc}
 & \multicolumn{2}{c|}{Denoising} & \multicolumn{2}{c}{Deconvolution}\\
\midrule
$\tau$ & MYULA & Explicit & MYULA & Explicit\\
\midrule
$10^{-5}$ & 60.1& 0.71&    46.5&    1.4\\
$10^{-6}$ & 60.6& 0.71&    46.7&    1.4\\
\end{tabular}
\vspace*{0.2cm}
\caption{Average computation times for the imaging examples in seconds for 1000 iterations for different methods and step sizes $\tau$. We show the times for computing a single Markov chain for the experiments with the peppers image. Simulations were performed on the CPU using 24 AMD Ryzen 9 3900X 12-Core Processors.}
\label{table:computation_times_imaging}
\end{table}

\subsubsection{Image deconvolution}
As a last experiment we consider image deconvolution. That is, the potential $U$ is defined as $U(x) = \frac{1}{2\sigma^2}\|k*x-y\|^2 + \delta \|x-y\|^2 + \theta\TV(x)$. Since the convolution has a non-trivial kernel, in order to obtain a strongly convex $F(x)$ we add a small quadratic penalty $\delta \|x-y\|^2$ with $\delta=1e-3$. We set $y=k*y_0+n$ with $y_0$ the clean ground truth image, $k\in\R^{5\times 5}$ a Gaussian kernel with standard deviation 1 and $n$ Gaussian noise with standard deviation $\sigma=0.05$. The regularization parameter is set to $\theta=20$. Again we use a burnin phase of 5e5 iterations and afterwards use the following 5e5 samples to approximate expected value and variance. As a \emph{ground truth} to compare to, we run the explicit scheme with a Metropolis-Hastings correction step which ensures direct convergence to the true target density \cite{roberts1996exponential}, \cite[Chapter 6]{robert1999monte}. There, we use a burnin phase of 1e6 iterations and the following 1e6 samples in order to be even closer to the true target. We show again the visual results in \Cref{fig:image_deconvolution_results}, the convergence behavior of expected value and variance in \Cref{fig:image_deconvolution_convergence}, and the computation times in \Cref{table:computation_times_imaging} with similar behavior as in the case of denoising.
\begin{figure}[h]
    \centering
    \begin{subfigure}{\textwidth}
    \centering
    \includegraphics[height=2cm,valign=t]{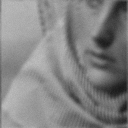}%
    \includegraphics[height=2cm,valign=t]{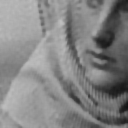}%
    \includegraphics[height=2cm,valign=t]{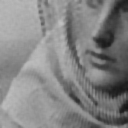}%
    \includegraphics[height=2cm,valign=t]{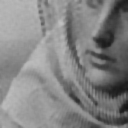}%
    \includegraphics[height=2cm,valign=t]{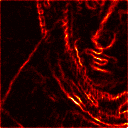}%
    \includegraphics[height=2cm,valign=t]{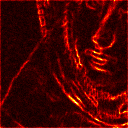}%
    \includegraphics[height=1.9cm,trim={2.72cm 1.8cm 0 1.5cm},valign=t]{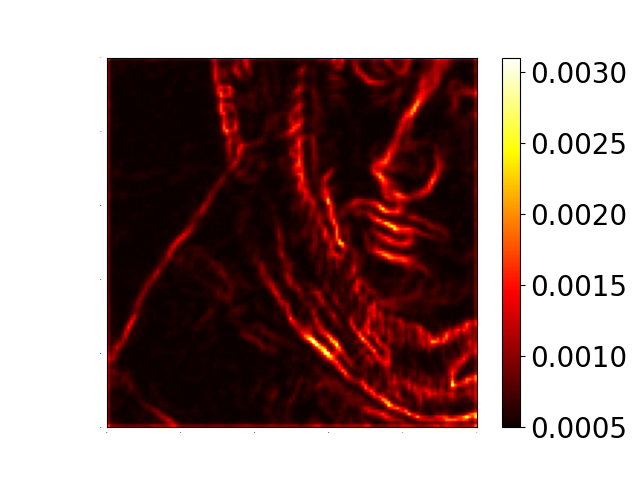}
    \end{subfigure}\\
    \begin{subfigure}{\textwidth}
    \centering
    \includegraphics[height=2cm]{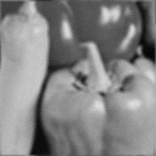}%
    \includegraphics[width=2cm]{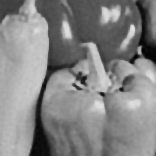}%
    \includegraphics[height=2cm]{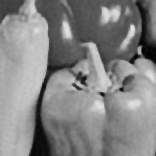}%
    \includegraphics[width=2cm]{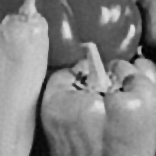}%
    \includegraphics[width=2cm]{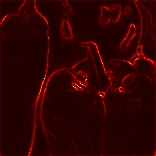}%
    \includegraphics[height=2cm]{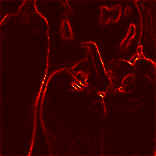}%
    \includegraphics[width=2.88cm,trim={3cm 1.5cm 0 1.5cm},clip]{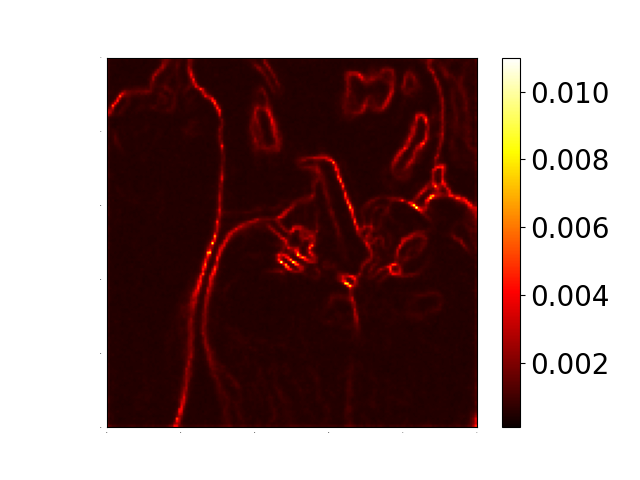}
    \end{subfigure}
    \caption{Deconvolution: estimated expected values and variances. From left to right: Corrupted image $y$, expected values computed with BP, MYULA, and the proposed explicit scheme. Then variances computed with BP, MYULA, and the proposed explicit scheme. For MYULA and the proposed method we compute the statistics using the following 5e5 iterates after a burnin phase of 5e5 iterations.}
    \label{fig:image_deconvolution_results}
\end{figure}
\begin{figure}[h]
    \centering
    \includegraphics[width=0.49\textwidth]{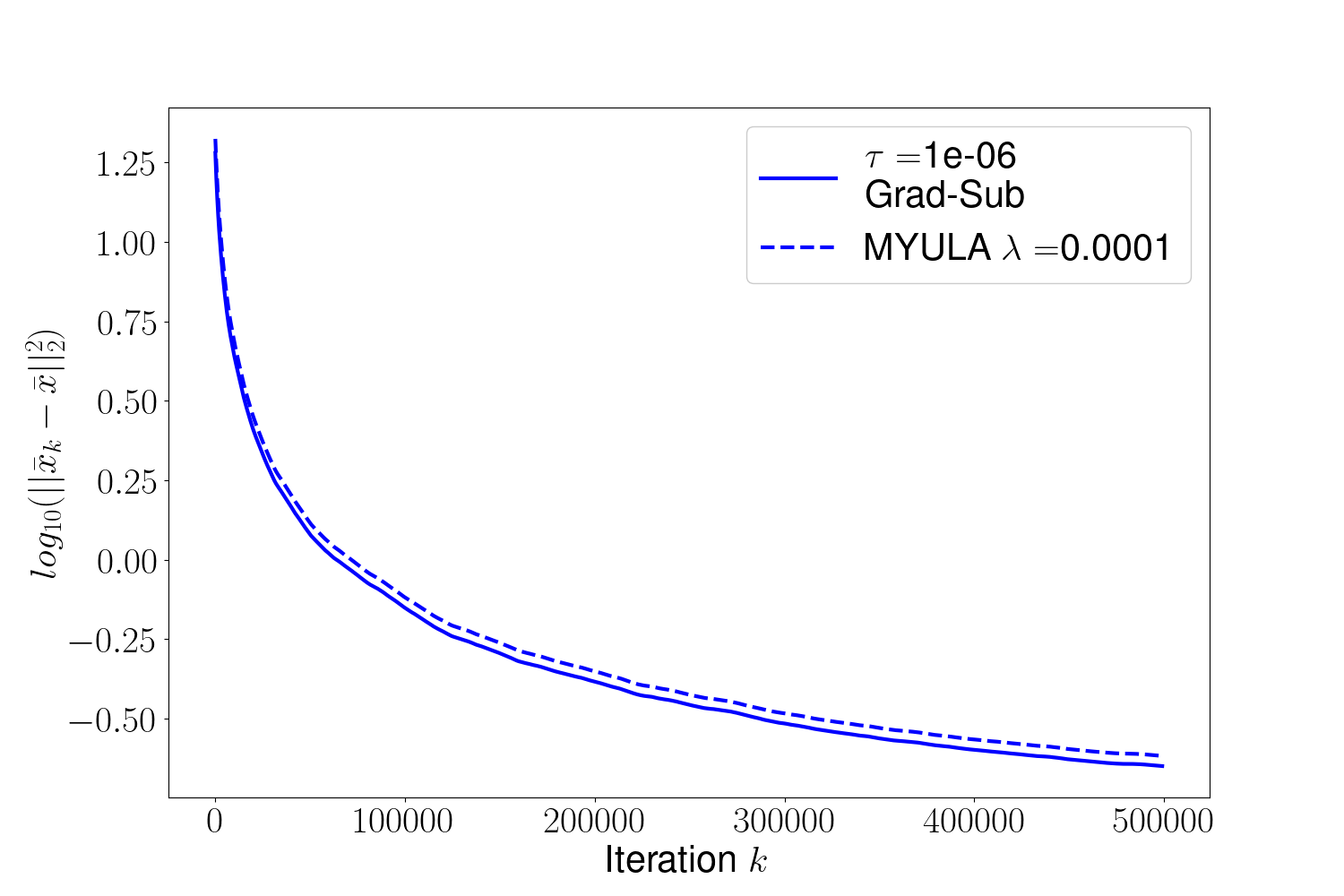}
    \includegraphics[width=0.49\textwidth]{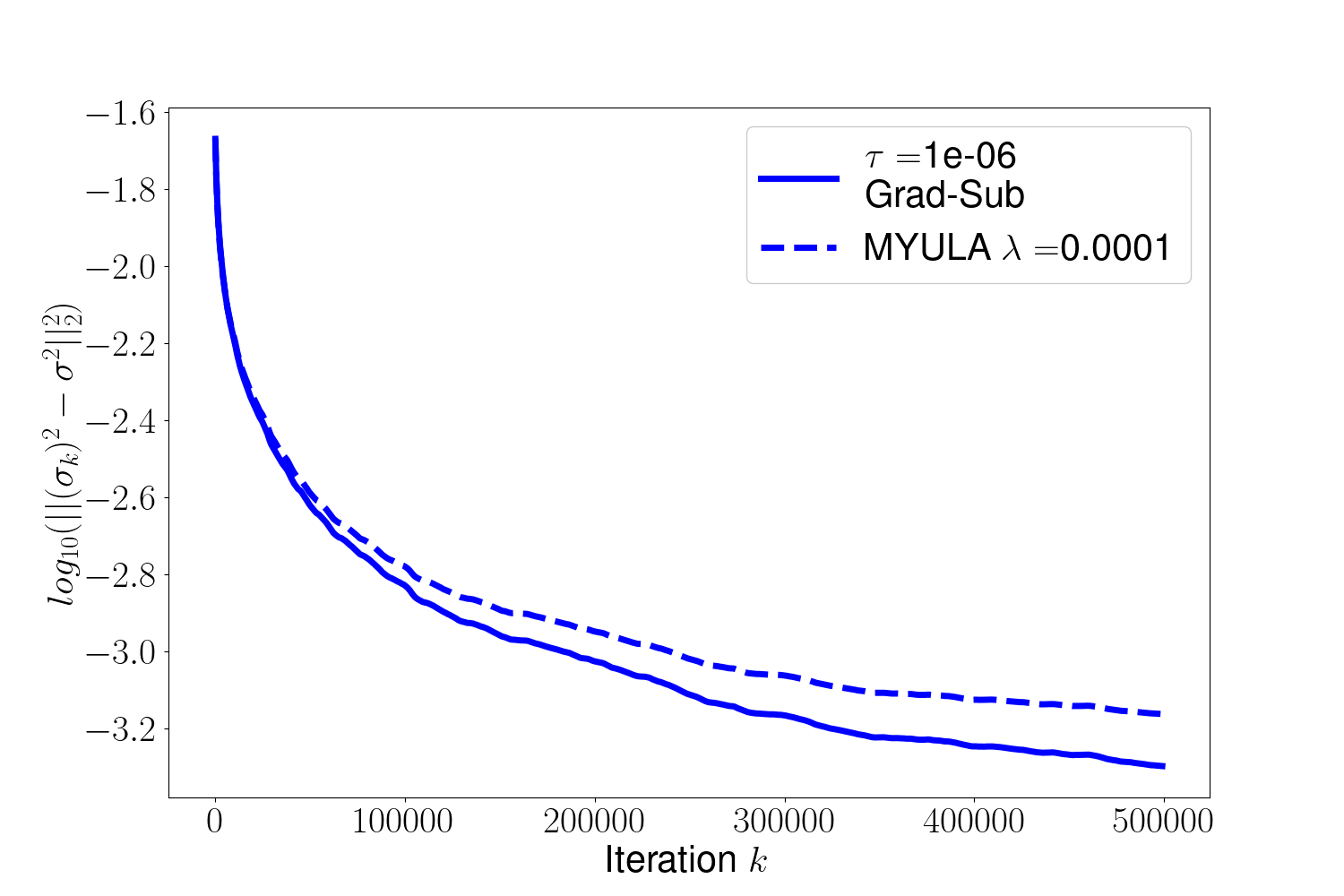}
    \caption{Denoising: $L_2$-error between estimated expected value (left) and variance (right) of the proposed explicit scheme and MYULA each compared to BP results for the peppers image. We use a burnin phase of 5e5. The symbols $\bar{x}_k,\sigma_k$ denote the emprical expected value and variance using $k$ successive iterates, $\bar{x},\sigma$ the estimates from BP.}
    \label{fig:image_deconvolution_convergence}
\end{figure}
\section{Conclusion}\label{sec:conclusion}
We have analyzed subgradient Langevin dynamics for sampling from densities induced by non-smooth potentials. We have proven exponential ergodicity in the time-continuous setting as well as geometric ergodicity of an explicit and a semi-implicit discretization scheme. Moreover, the discretizations are proven to satisfy the law of large numbers which allows for using consecutive iterates in order to compute statistics of the target density when simulating Markov chains. Lastly, numerical experiments confirm the theoretical findings.
\paragraph{Limitations and future work}
We aim to extend the results to functions which are strongly convex only outside a compact set. Moreover, we plan to extend the non-smooth analysis to kinetic Langevin dynamics.
\paragraph{Acknowledgments}
We want to thank Prof. Thomas Pock for providing helpful feedback for the article and Prof. Daniel Rudolf and Julian Hofstadler for pointing out additional references, which helped improve our work.

\bibliographystyle{siamplain}
\bibliography{references}

\ifarXiv
    \foreach \x in {1,...,\numbersupplementpages}
    {
        \clearpage
        \includepdf[pages={\x}]{\supplementfilename}
    }
\fi

\makeatletter\@input{xx.tex}\makeatother

\end{document}


\maketitle

\section{Implementation details}\label{sec:implementation_details}
In the following we provide some more details about the numerical implementation. We want to emphasize once again that the code is available at~\cite{habring2024subgrad_git}.

All simulations were performed on the CPU on a machine with 24 AMD Ryzen 9 3900X 12-Core Processors.

\subsection{Computation of total variation distance and Kullback-Leibler divergence}\label{sec:computation_of_metrics}
\paragraph{Wasserstein distance}
The Wasserstein-2 distances for the 2D experiments are estimated using the package Python Optimal Transport (POT)~\cite{flamary2021pot,flamary2024pot}.
Specifically, to estimate the Wasserstein-2 distance $\Wc_2(\mu_k^\tau,\pi)$ we simulate $N=1000$ independent chains $(X_{k,n}^\tau)_{k,n}$ where $k$ denotes the iteration and we added the subscript $n$ to index the independent chains. The distribution $\mu_k^\tau$ is then approximated by the empirical distribution 
\begin{equation}\label{S:eq:empricial_dist}
    \mu_k^\tau \approx\hat \mu_k^\tau = \frac{1}{N}\sum_{n=1}^N\delta_{X_{k,n}^\tau}
\end{equation}
where $\delta_x$ denotes the Dirac distribution in $x\in\R^d$. On the other hand, to obtain a ground truth target, we first approximate the partition function
\[
    Z=\int e^{-U(x)}\d x
\]
via numerical integration over $x\in [-10,10]^2$. Then we approximate the target $\pi$ by a piecewise constant density for which we grid the domain $[-10,10]^2$ with $200$ bins in both coordinates and afterward integrate $\pi$ over each bin to obtain the bin's probability. Then we approximate $\pi$ via
\begin{equation}\label{S:eq:empricial_dist_target}
    \pi \approx\hat \pi = \sum_{i,j=1}^{200}\delta_{x_{i,j}} p_{i,j}
\end{equation}
where $x_{i,j}$ is the center of bin $(i,j)$ and $p_{i,j}$ its probability. The Wasserstein distance between $\mu_k^\tau$ and $\hat \pi$ is then approximated using the POT function \texttt{emd}.
\paragraph{Total variation}
In order to compute the total variation distance, for the target $\pi$ we again compute the approximation $\hat\pi$ from~\eqref{S:eq:empricial_dist_target}. Then, based on the $N$ independent chains, we compute a histogram from $(X_{k,n}^\tau)_{n}$ with \texttt{numpy.histogram2d} for which we use the same bins as for $\hat\pi$. For these two empirical distributions the total variation distance can then be computed exactly as the L1 distance between the densities.

\section{Convergence in physical time}\label{sec:physical_time}
In this section we show the same plots as in \Cref{M-sec_2d_experiments} but over the physical time instead of the number of iterations. That is, for the chain with step size $\tau$ we plot the curve $\{(k\tau,X_k^\tau)\;|\;k\}$. As can be observed in \Cref{S:fig:without_K_convergence,S:fig:with_K_convergence} the Markov chains approximate an underlying continuous time process.
\begin{figure}[h]
    \centering
    \includegraphics[width=0.75\linewidth]{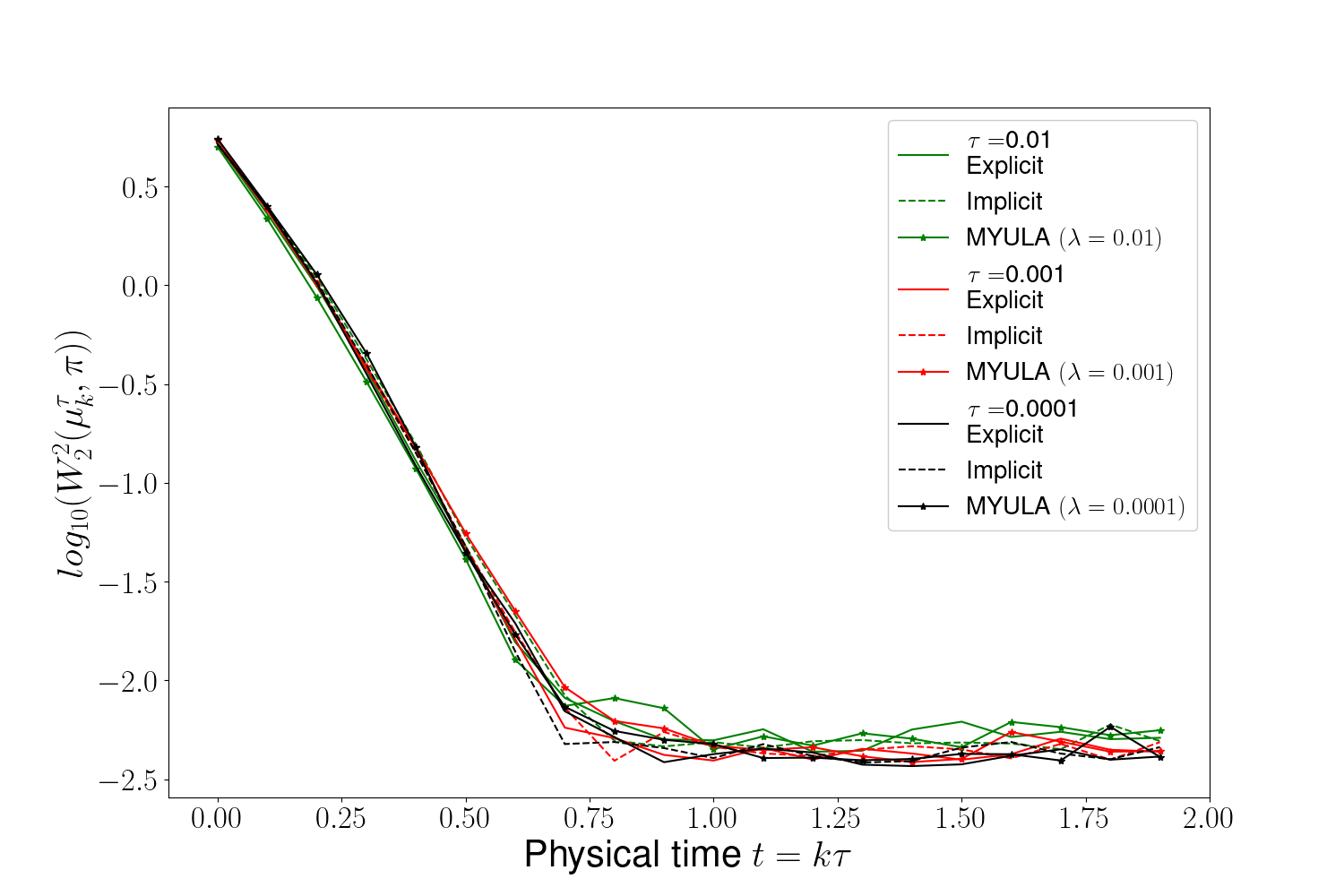}
    \includegraphics[width=0.75\linewidth]{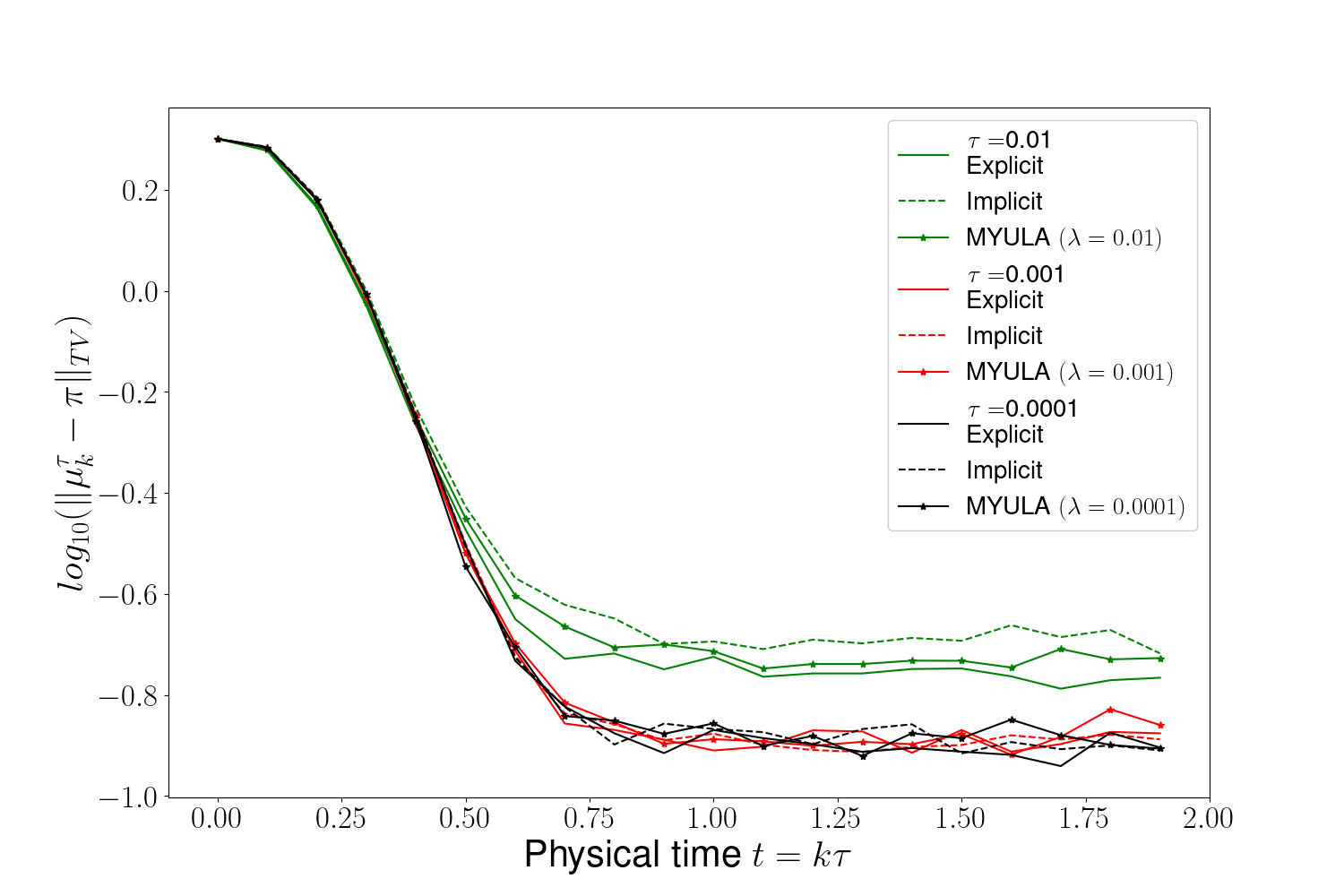}
    \caption{Potential $U(x)=F(x)+G(x)$. On the top we show the Wasserstein-$2$ and on the bottom the total variation distance between samples and target $\pi$ for different step sizes and methods. On the x-axis we show the physical time corresponding to the iteration.}
    \label{S:fig:without_K_convergence}
\end{figure}

\begin{figure}[h]
    \centering
    \includegraphics[width=0.75\linewidth]{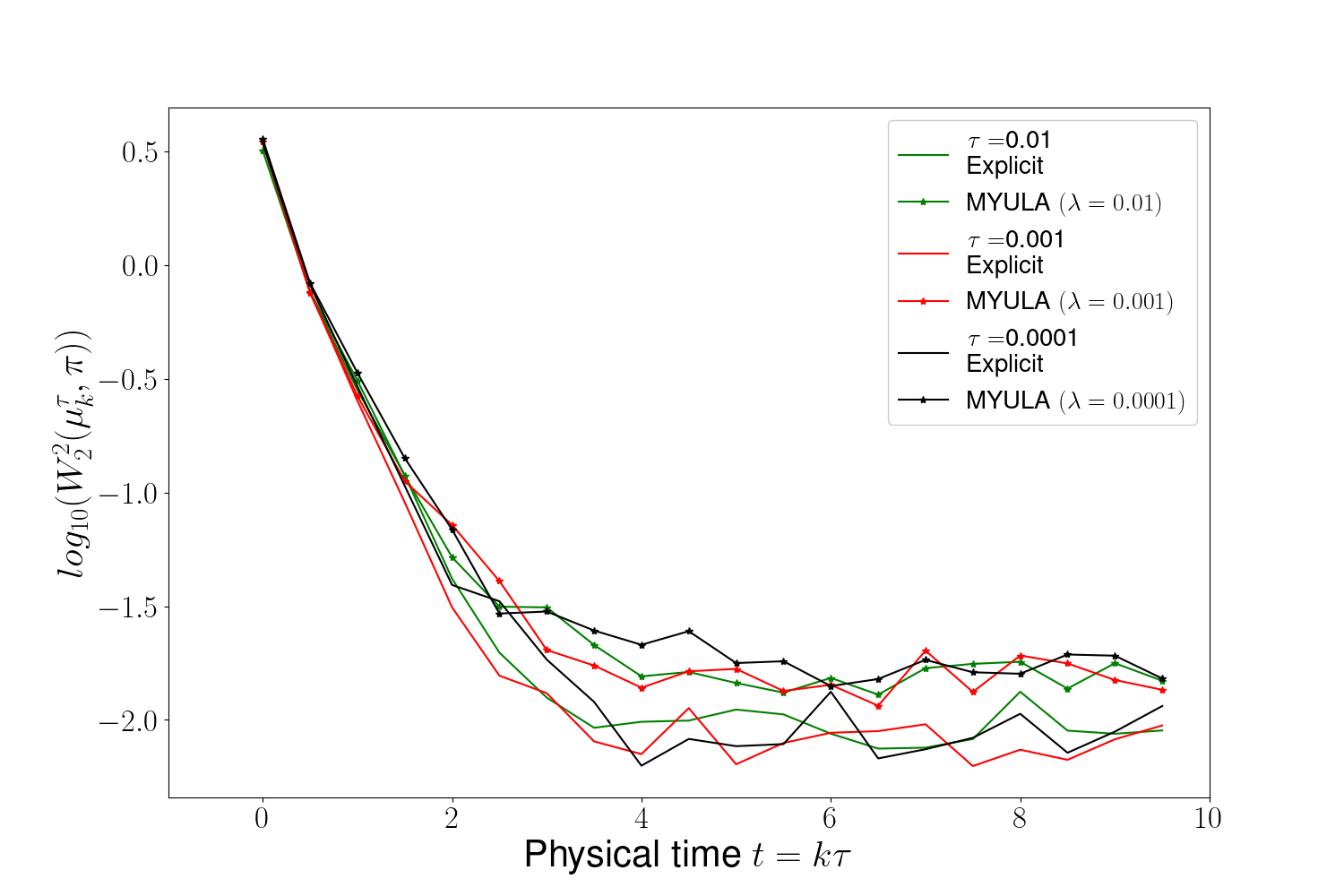}
    \includegraphics[width=0.75\linewidth]{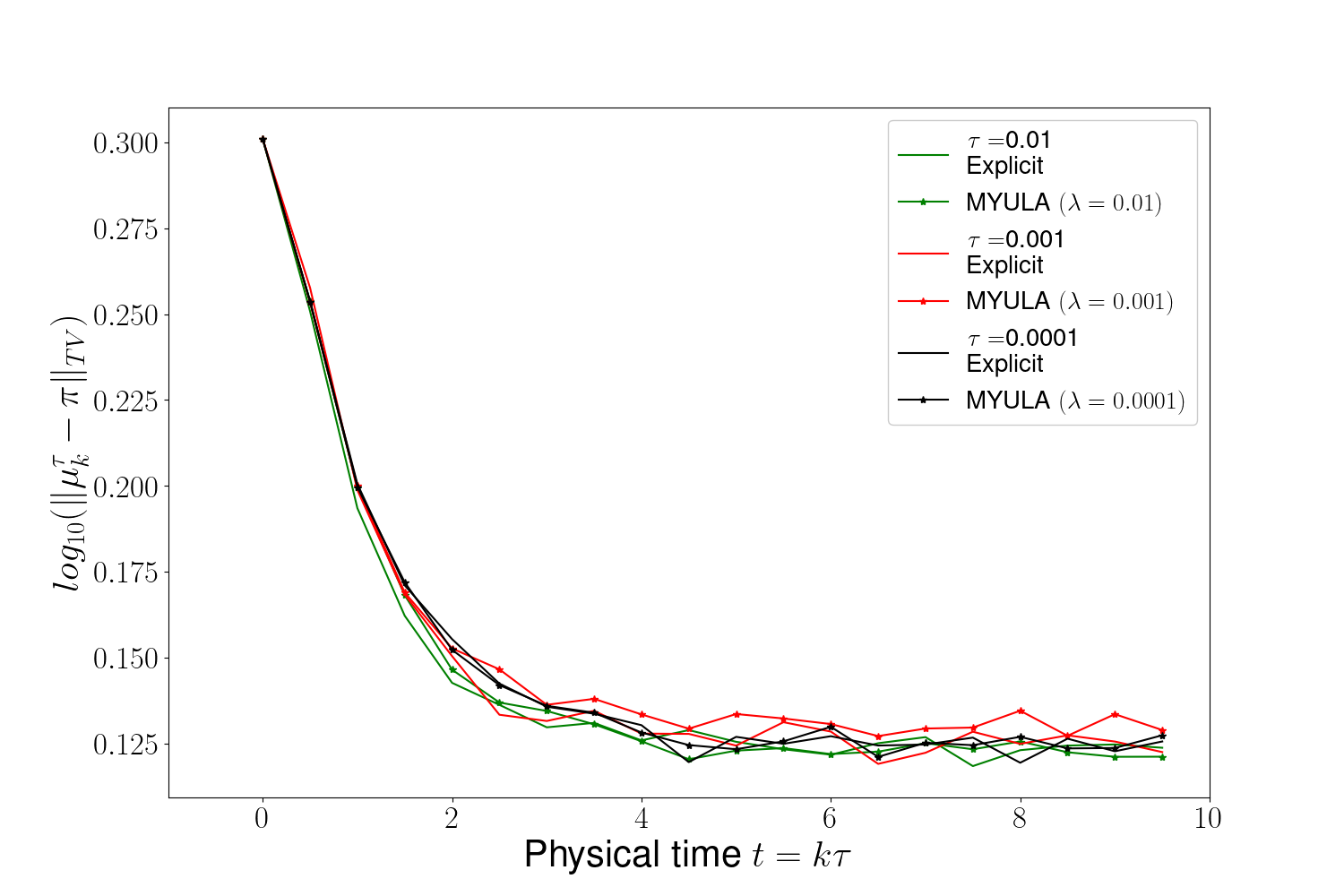}
    \caption{Potential $U(x)=F(x)+G(x)$. On the top we show the Wasserstein-$2$ and on the bottom the total variation distance between samples and target $\pi$ for different step sizes and methods. On the x-axis we show the physical time corresponding to the iteration.}
    \label{S:fig:with_K_convergence}
\end{figure}
\section{$V$-norm is weighted total variation}
In the following we provide a short proof of the (well-known) fact that the $V$-norm is, indeed, a weighted total variation norm.
\begin{lemma}\label{lemma:vnorm}
    Let $V:\R^d\rightarrow (0,\infty)$ and define for any Radon measure $\mu$ on $\Bc(\R^d)$ 
    \begin{equation}
        \|\mu\|_V=\sup_{g\leq V}\int g(x)\d \mu(x).
    \end{equation}
    Then it holds for any finite Radon measure $\mu$ on $\Bc(\R^d)$ with $\int V(x)\d |\mu|(x)<\infty$
    \begin{equation}
        \|\mu\|_V = \int V(x)\d |\mu|(x).
    \end{equation}
\end{lemma}
\begin{proof}
    We trivially have the inequality
    \begin{equation}
        \sup_{|g|\leq V}\int g(x) \d \mu(x) 
        \leq \int V(x) \d |\mu|(x).
    \end{equation}
    Conversely, since $V>0$ we find
    \begin{equation}
        \begin{aligned}
            \sup_{|g|\leq V}\int g(x) \d \mu(x) 
            =& \sup_{g\leq V} \int \frac{|g(x)|}{V(x)} V(x)\d \mu(x)\\
            =& \sup_{|f|\leq 1} \int f(x) V(x)\d \mu(x)\\
            \geq& \int V(x) \d |\mu|(x)
        \end{aligned}
    \end{equation}
    which follows from the fact that the total variation norm is the dual norm of the supremum norm on the space of continous functions vanishing at infinity~\cite[Remark 1.57]{ambrosio2000functions}.
\end{proof}

\bibliographystyle{siamplain}
\bibliography{references}